\documentclass{amsart}

\usepackage{amssymb}
\usepackage{amsthm}
\usepackage{graphicx}
\usepackage{hyperref}
\usepackage{multirow}
\usepackage{array}
\usepackage{comment}
\usepackage[colorinlistoftodos]{todonotes}
\usepackage[english]{babel}
\usepackage{lmodern}

\usepackage[utf8]{inputenc}
\usepackage[T1]{fontenc}
\usepackage{microtype}

\DeclareMathOperator{\NS}{NS}
\DeclareMathOperator{\Spec}{Spec}
\DeclareMathOperator{\disc}{disc}
\DeclareMathOperator{\rank}{rank}
\DeclareMathOperator{\charac}{char}
\DeclareMathOperator{\spec}{sp}
\DeclareMathOperator{\rad}{rad}

\newcommand{\modulo}[3]{#1\equiv#2\textrm{ }(\textrm{mod }#3)}

\newtheorem{theorem}{Theorem}[section]
\newtheorem{proposition}[theorem]{Proposition}
\newtheorem{lemma}[theorem]{Lemma}
\newtheorem{corollary}[theorem]{Corollary}
\newtheorem{question}[theorem]{Question}

\newtheorem{maintheorem}[theorem]{Main Theorem}

\theoremstyle{definition}
\newtheorem{definition}[theorem]{Definition}
\newtheorem{example}[theorem]{Example}
\newtheorem{remark}[theorem]{Remark}

\begin{document}

%\mathclass{Primary 14H52; Secondary 11D25 11D45 11G05.}

\title{Distribution of Mordell--Weil ranks of families of elliptic curves}

\author{Bartosz Naskręcki}
\address{Faculty of Mathematics and Computer Science, Adam Mickiewicz University\\
	Umultowska 87, 61-614 Poznań, Poland\\
	and School of Mathematics, University of Bristol, University Walk, Bristol BS8 1TW, UK\\
	E-mail: nasqret@gmail.com}

\maketitle

\begin{abstract}
We discuss the distribution of Mordell--Weil ranks of the family of elliptic curves $y^2=(x+\alpha f^2)(x+\beta b g^2)(x+\gamma h^2)$ where $f,g,h$ are coprime polynomials that parametrize the projective smooth conic $a^2+b^2=c^2$ and $\alpha,\beta,\gamma$ are elements from $\overline{\mathbb{Q}}$. In our previous papers we discussed certain special cases of this problem and in this article we complete the picture by proving the general results.
\end{abstract}

\nocite{MAGMA}
\section{Introduction}
In our previous papers \cite{Naskrecki_Acta}, \cite{NaskreckiMW} and thesis \cite{Naskrecki_thesis} we have studied in several aspects the following family of curves
\begin{equation}\label{eq:general_family}
y^2=x(x-f^2)(x-g^2)
\end{equation}
where $f,g$ are two elements from certain field. We considered the case where $f,g$ are rational functions in one variable and satisfy  the extra relation $f^2+g^2=h^2$ where $h$ is also a rational function. The most general results were obtained in the case when $f,g,h\in\overline{\mathbb{Q}}[t]$ under the assumption that the polynomials are pairwise coprime. By a simple change of variables we can put this equation into a more symmetric form
\[y^2=(x+f^2)(x+g^2)(x+h^2)\]
Now we can either forget the relation $f^2+g^2=h^2$, but then the connection with previous equation is lost or we can generalize it in another direction. In this article we consider a general family of the form
\begin{equation}\label{eq:general_abc_family}
y^2=(x+\alpha f^2)(x+\beta g^2)(x+\gamma h^2). 
\end{equation}
This curve contains an obvious point $(0,f g h)$ of infinite order and we would like to discuss how the rank of the Mordell-Weil group of curve \eqref{eq:general_abc_family} varies with  $\alpha,\beta,\gamma$. 

We can offer the most complete description in the situation when 
\begin{equation}\label{eq:degree_2_assumption}
\max\{\deg f,\deg g,\deg h\}=2.
\end{equation}
\noindent
We will show that in such cases what we obtain is a Weierstrass model of a generic fiber of a K3 surface defined over some number field. We are interested in the computation of the Mordell--Weil group of the generic fiber over the rational function field with suitable coefficients. This will be related to the computation of the N\'{e}ron-Severi group  of the associated elliptic surface.

\noindent
The main result of this paper will be proved in Sections \ref{sec:gen_gamilies} and \ref{sec:case_1_1_1}.
\begin{maintheorem}
Let $(\alpha,\beta,\gamma)\in\overline{\mathbb{Q}}^{3}$ be such that $\alpha\beta\gamma\neq 0$. Assume that $f,g,h$ are coprime polynomials in $\overline{\mathbb{Q}}[t]$ that satisfy conditions $f^2+g^2=h^2$ and \eqref{eq:degree_2_assumption}. Then the curve \eqref{eq:general_abc_family} is smooth and is an elliptic curve. The Mordell-Weil rank of  elliptic curve \eqref{eq:general_abc_family} over $\overline{\mathbb{Q}}(t)$ varies between $2$ and $6$ and each case is explicitly described in Sections \ref{sec:gen_gamilies} and \ref{sec:case_1_1_1}.
\end{maintheorem}

After short preliminaries in Section \ref{sec:prelimin} we analyse the rank variation in family \ref{eq:general_abc_family} in Section \ref{sec:gen_gamilies}. Then we focus on the special case of the family with constants $\alpha=\beta=\gamma=1$ in Section \ref{sec:case_1_1_1} that was previously discussed in \cite{Naskrecki_Acta}, \cite{Naskrecki_thesis}, \cite{NaskreckiMW}. We develop a certain mechanism that allows us to easily switch between triples $f,g,h$ that parametrize a conic $a^2+b^2=c^2$. In fact, we prefer to perform certain computations in the most convenient way, by choosing the \emph{standard} triple $f=t^2-1,g=2t,h=t^2+1$. Such a choice is arbitrary, and we can switch to any other such triple of polynomials. This implies in particular that geometrically the elliptic surface that corresponds to the curve \eqref{eq:general_abc_family} and the one that corresponds to the choice of standard polynomials is isomorphic (not as fibred surface). So we can analyse simply the family
\begin{equation}\label{eq:standard_abc_family}
E_{(\alpha,\beta,\gamma)}:\ y^2=(x+\alpha (t^2-1)^2)(x+\beta\cdot 4t^2)(x+\gamma (t^2+1)^2). 
\end{equation}
In paper \cite{NaskreckiMW} we have analyzed mostly the case $\alpha=\beta=\gamma=1$ with full description of the geometric Mordell-Weil group and certain results over number fields. We have also studied certain base changes of the family induced by a map $\phi:\mathbb{P}^{1}\rightarrow\mathbb{P}^{1}$, $\phi:t\mapsto \phi(t)$. Further we concentrate on the reduction modulo a prime of the model $\alpha=\beta=\gamma=1$ and develop the properties of the supersingular K3 surfaces that occur at certain primes.

The curves in family \eqref{eq:general_abc_family} appeared in \cite{UlasBremner} and were studied from another perspective in \cite{Naskrecki_EDS}. They might also have applications in the study of dynamics on supersingular K3 surfaces but this will be analysed elsewhere, cf. \cite{Esnault}, \cite{Shimada_auto_K3}.

\section{Preliminaries}\label{sec:prelimin}
We formulate in this section the necessary definitions and theorems that will be used throughout the article. The reader can consult also \cite{Shioda_Mordell_Weil}, \cite{Shioda_Schutt}, \cite{Luijk_Heron}.
\begin{definition}\label{definition:ell_surface}
Let $k$ be an algebraically closed field. Let $C$ be a smooth projective curve over $k$ and $S$ be a smooth projective surface over $k$. We call a triple $(S,C,\pi)$ an elliptic surface when $\pi:S\rightarrow C$ is a surjective morphism such that
\begin{itemize}
\item there exists a non-empty set $B\subset C(k)$ such that for any $v\in C(k)\setminus B$ the fibre $\pi^{-1}(v)$ is a curve of genus $1$,
\item there exists a section $O: C\rightarrow S$ of the morphism $\pi$,
\item no fibre $\pi^{-1}(v)$ for $v\in C(k)$ contains $(-1)$-curves.
\end{itemize}
\end{definition}

To any elliptic curve over $F(t)$ we can attach the corresponding elliptic surface fibred over $\mathbb{P}^{1}_{F}$. We call it a Kodaira-N\'{e}ron model of $E$ over $F(t)$.

For an elliptic curve $E$ over $K=\overline{\mathbb{Q}}(t)$ we denote by $\langle \cdot,\cdot\rangle_{E}$ the height pairing attached to $E$ as in \cite{Shioda_Mordell_Weil}. The group $E(K)/E(K)_{\textrm{tors}}$ with the induced pairing $\langle\cdot,\cdot\rangle_{E}$ is a positive definite lattice, cf. \cite[Theorem 7.4]{Shioda_Mordell_Weil}. To simplify the notation, we write $\langle\cdot,\cdot\rangle$ if the curve $E$ is fixed. Explicitly, for two given points $P,Q\in E(K)$ their intersection pairing is given by
\begin{equation}\label{equation:height_pairing_formula}
\langle P,Q\rangle = \chi(S)+\overline{P}.\overline{O}+\overline{Q}.\overline{O}-\overline{P}.\overline{Q}-\sum_{v\in B}c_{v}(P,Q). 
\end{equation}
For a point $P$ in $E(K)$ we denote by $\overline{P}$ the curve which lies in $S$ and is the image of a section determined by the point $P$, cf. \cite[Lemma 5.2]{Shioda_Mordell_Weil}. The curve $\overline{O}$ is the image of the zero section $O:\mathbb{P}^{1}_{\overline{\mathbb{Q}}}\rightarrow S$. In the case $P=Q$ the formula simplifies to
\begin{equation}\label{equation:height_formula}
\langle P,P\rangle = 2\chi(S)+2\overline{P}.\overline{O}-\sum_{v\in B}c_{v}(P,P). 
\end{equation}
The rational numbers $c_{v}(P,Q)$ depend only on the fibre type above $v\in B$ of bad reduction and on the indices of the components that intersect curves $\overline{P}$ and $\overline{Q}$, cf. \cite[Theorem 8.6]{Shioda_Mordell_Weil}. We usually denote $c_{v}(P,P)$ by $c_{v}(P)$. We denote by $\langle P,P\rangle$ the \textit{height of point} $P$.

For an elliptic surface $(S,C,\pi)$ we denote by $\NS(S)$ the N\'{e}ron-Severi group of the surface $S$. It follows from \cite[Cor. 3.2]{Shioda_Mordell_Weil} that it is a finitely generated and torsion free abelian group. We denote its rank by $\rho(S)$ and call it the \emph{Picard number}. Moreover, we can induce on $\NS(S)$ a structure of a lattice by using the intersection product of divisors. Let $B$ denote the set of closed points $v$ in $C$ such that $\pi^{-1}(v)$ is singular. We denote by $m_{v}$ the number of components of $\pi^{-1}(v)$. We denote by $E$ the generic fiber of $\pi$ treated as an elliptic curve over $k(C)$. The Mordell-Weil rank of the generic fiber $E(k(C))$ and the rank $\rho(S)$ are related by the so-called Shioda-Tate formula

\begin{theorem}[\protect{\cite[Cor. 5.3]{Shioda_Mordell_Weil}}]
Let $(S,C,\pi)$ be an elliptic surface with generic fiber $E$. The following equality holds
\begin{equation}
\rho(S)=2+\sum_{v\in B}(m_{v}-1)+\rank E(k(C)).
\end{equation}
\end{theorem}
It is standard to give upper bounds for Picard numbers in terms of the Euler characteristic $\chi(S)=\chi(S,\mathcal{O}_{S})$ and the genus $g(C)$ of curve C. In characteristic zero they follow from Lefschetz (1,1)-classes theorem \cite[Prop. 3.3.2]{Huybrechts_Geometry}.
\begin{equation}\label{eq:zero_char_bound}
\rho(S)\leq 10 \chi(S)+2g(C). 
\end{equation}
In positive characteristic we have a weaker bound
\begin{equation}\label{eq:pos_char_bound}
\rho(S)\leq 12\chi(S)-2+4g(C). 
\end{equation}
The numbers $\chi(S)$ of elliptic surface $S$ can be used to identify the type of algebraic surface represented by $S$. An elliptic surface $S$ is a rational surface if and only if $\chi(S)=1$, it is a K3 surface when $\chi(S)=2$ and is of Kodaira dimension when $\chi(S)\geq 3$. We can compute the numbers $\chi(S)$ in terms of explicit numbers $e(F_{v})$ which depend on the reduction types of bad fibers $\pi^{-1}(v)$.
\begin{theorem}[\protect{\cite[Thm. 1]{Oguiso_c2}}]
Let $(S,C,\pi)$ be an elliptic surface over the field $k$ of characteristic $\charac(k)\neq 2,3$. The following equality holds
\[12\chi(S)=\sum_{v\in B}e(F_{v})\]
where the number $e(F_{v})$ depend on the Kodaira type of fiber $\pi^{-1}(v)$ as follows
{
\renewcommand{\arraystretch}{1.4}

\begin{table}[ht]
\begin{center}
\[\begin{array}{|c|c|c|c|c|c|c|c|c|}
\hline
\textrm{Typ }F_{v} & I_{n} (n\geq 1) & II & III & IV & I_{n}^{*} (n\geq 0) & II^{*} & III^{*} & IV^{*}\\
\hline
e(F_{v}) & n & 2 & 3 & 4 & n+6 & 10 & 9 & 8 \\
\hline
\end{array}\]
\end{center}
\caption{Euler numbers $e(F_{v})$.}\label{tab:Euler_num}
\end{table}}
\end{theorem}

In general the bounds \eqref{eq:zero_char_bound}, \eqref{eq:pos_char_bound} are not sharp and we need the approach that requires the use of $\ell$-adic cohomology and good reduction to positive characteristic. The main ideas come from \cite[Ex. 20.3.6]{Fulton}. The details of this approach are explained in \cite[\S 6]{Luijk_Heron} and \cite[\S 4]{Kloosterman_rank_15}. We say that our elliptic surface $(S,C,\pi)$ has a model over $\Spec A$ where $A$ is a discrete valuation ring in some number field $K$, with maximal ideal $\mathfrak{p}$ with residue field $\mathbb{F}_{q}=A/\mathfrak{p}$ of characteristic $p$. We assume that $S$ has \emph{good reduction} modulo $p$ which means that we have a smooth morphism $S\rightarrow\Spec A$. For any prime $\ell\neq p$ there are natural injective homomorphisms
\begin{equation}
\NS(S_{\overline{\mathbb{Q}}})\otimes\mathbb{Q}_{\ell}\hookrightarrow\NS(S_{\overline{\mathbb{F}}_{q}})\otimes\mathbb{Q}_{\ell}\hookrightarrow H^2_{\textrm{\'{e}t}}(S_{\overline{\mathbb{F}}_{q}},\mathbb{Q}_{\ell}(1)).
\end{equation}
On the $\ell$-adic cohomology group $H^2_{\textrm{\'{e}t}}(S_{\overline{\mathbb{F}}_{q}},\mathbb{Q}_{\ell}(1))$ there is an action of Frobenius automorphism $\Phi$. Its characteristic polynomial $P(x)=\det(Ix-\Phi)$ has the property that it is defined over $\mathbb{Z}$ and all its roots are complex algebraic numbers of norm $q$. We denote the multiplicity of root $\alpha$ by $\lambda(\alpha,\Phi)$. We denote by $R_{\Phi}$ the sum of multiplicities $\lambda(\zeta q,\Phi)$ where $\zeta$ is some root of unity.
\begin{corollary}[\protect{\cite[Cor. 2.3]{Luijk_Heron}}]\label{cor:spec_NS}
For the elliptic surface $(S,C,\pi)$ with good reduction at $p$ we have the inequalities
\begin{equation}
\rho(S_{\overline{\mathbb{Q}}})\leq \rho(S_{\overline{\mathbb{F}}_{q}})\leq R_{\Phi}.
\end{equation}
\end{corollary}
In practical terms we deal with the case where the base curve $C$ is the projective line $\mathbb{P}^{1}$. The $\ell$-adic cohomology group $H^2_{\textrm{\'{e}t}}(S_{\overline{\mathbb{F}}_{q}},\mathbb{Q}_{\ell}(1))$ is of dimension $12\chi(S_{\overline{\mathbb{F}}_{q}})-2$ in this case and the Frobenius automorphism acts on the N\'{e}ron-Severi group in an explicit way. The action on the part coming from the components of bad fibers only permutes the components in each fiber. On the part that comes from sections (which are induced by points on the generic fiber) the Frobenius action can be determined by the $x\mapsto x^{q}$ Frobenius action on the coefficients of points on the generic fiber.
To compute the characteristic polynomial of $\Phi$ we have to apply the Grothendieck-Lefschetz formula and count the points over finite fields, cf. \cite{Naskrecki_Acta}. When the image of specialization map of N\'{e}ron-Severi groups \[\spec_{\mathfrak{p}}:\NS(S_{\overline{\mathbb{Q}}})\rightarrow\NS(S_{\overline{\mathbb{F}}_{q}})\]
is of finite index by the properties of lattices we have that the discriminant $\Delta(\NS(S_{\overline{\mathbb{Q}}}))$ equals $[\NS(S_{\overline{\mathbb{F}}_{q}}):\spec_{\mathfrak{p}}(\NS(S_{\overline{\mathbb{Q}}}))]^2 \Delta(\NS(S_{\overline{\mathbb{F}}_{q}}))$. The discriminant $\Delta(\NS(S_{\overline{\mathbb{F}}_{q}})$ can be computed if we assume the Artin-Tate conjecture. This is unconditional in the case when $S$ is a K3 surface, cf. \cite[Thm. 6.1]{Milne_Tate_Conjecture}, \cite[Thm. 5.2]{Artin_SD_K3_Tate_Conjectures}, \cite[Thm. 5.2]{Naskrecki_Acta}. We will use it only in the form of the following corollary.
\begin{corollary}\label{cor:disc_mod_p}
Let $(S,\mathbb{P}^{1},\pi)$ be a K3 elliptic surface with good reduction at prime $p$. We assume that the N\'{e}ron-Severi group is defined over $\mathbb{F}_{q}$. Then
\begin{equation}
\modulo{q^{\rho(S_{\mathbb{F}_{q}})-21}\cdot\frac{\lim\limits_{x\rightarrow q} P(x)}{(x-q)^{\rho(S_{\mathbb{F}_{q}})}}}{-\Delta(\NS(S_{\mathbb{F}_{q}}))}{(\mathbb{Q}^{\times})^2}.
\end{equation}
\end{corollary}
All over the paper we work with Weierstrass models of elliptic curves over $\overline{\mathbb{Q}}(t)$. We say that such a model is minimal at $t-a$ for $a\in\overline{\mathbb{Q}}$ if it is minimal in the usual sense, cf.\cite[Chap. VII.1]{Silverman_arithmetic}. We say also that it is \emph{minimal at infinity} or at $t=\infty$ if the change of coordinates $t\mapsto 1/s$ and $(x,y)\mapsto (x s^{2n},y s^{3n})$ for some choice of $n$ provide us with a Weierstrass model which is minimal at $s=0$ in the usual sense. If a model is minimal at every place we say it is globally minimal. In general we can check if a prime $p$ is a prime of good reduction for an elliptic surface over $\mathbb{P}^{1}$ by analysing its globally minimal Weierstrass model. It is particularly simple for curves defined over $\mathbb{Q}(t)$. We denote by $\rad(h)$ the square free part of polynomial $h$.
\begin{lemma}[\protect{\cite[Tw. 2.2.12]{Naskrecki_thesis}}]\label{lem:good_reduction}
Let $E$ be an elliptic curve defined over $\mathbb{Q}(t)$ with globally minimal Weierstrass model
\[y^2+a_{1}xy+a_{3}y=x^3+a_{2}x^2+a_{4}x+a_{6}\]
where $a_{i}\in\mathbb{Z}[t]$. Let $\Delta\in\mathbb{Z}[t]$ be its discriminant and $j=f/g\in\mathbb{Q}(t)$ its $j$-invariant where $f,g$ are coprime polynomials in $\mathbb{Z}[t]$. Let $p$ be a prime number greater than $5$ such that $p\nmid \disc(h)$ for every polynomial $h$ in the list $\{\rad(a_{i}):i=1,2,3,4,6\}\cup \{\rad(f),\rad(g),\rad(\Delta)\}$. We denote by $\tilde{E}$ the reduction modulo $p$ of the equation of $E$. If $\tilde{E}$ defines an elliptic curve over $\mathbb{F}_{p}(t)$ and the model is globally minimal and the reductions $\tilde{h}$ of polynomials $h\in\{\rad(f),\rad(g),\rad(\Delta)\}$ are separable, than the elliptic surface $\mathcal{E}$ with generic fiber $E$ has good reduction at $p$.
\end{lemma}

\section{General families with moderate ranks}\label{sec:gen_gamilies}
In this section we consider the variation of the Mordell-Weil rank in family \eqref{eq:standard_abc_family} for different choices of $(\alpha,\beta,\gamma)$. Our main tool in this section is the Shioda-Tate formula and theorems on good reduction of N\'{e}ron-Severi groups. Theorem \ref{thm:proper_params_change_coordinates} and a similar reasoning to Corollary \ref{cor:rank_comparison} allow us to reduce the computations on \eqref{eq:general_abc_family} to a fixed triple of coprime polynomials that establish a rational parametrization of the conic $a^2+b^2=c^2$. We choose the parametrizing polynomials $f=t^2-1$, $g=4t^2$ and $h=t^2+1$. We work over the field $\overline{\mathbb{Q}}(t)$ and a change of coordinates between two Weierstrass models \eqref{eq:standard_abc_family} for two different pairs $(\alpha,\beta,\gamma)$ and $(\alpha',\beta',\gamma')$ determines an equivalence relation between pairs: $(\alpha,\beta,\gamma)\sim (\alpha',\beta',\gamma')$ if and only if there exists an element $\lambda\in\overline{\mathbb{Q}}^{\times}$ such that $(\alpha,\beta,\gamma)=(\lambda\alpha',\lambda\beta',\lambda\gamma')$. This is equivalent to saying that the triples $(\alpha,\beta,\gamma)$, $(\alpha',\beta',\gamma')$ determine the same point in $\mathbb{P}^{2}(\overline{\mathbb{Q}})$. We can restrict to the affine part $\mathbb{A}^{2}(\overline{\mathbb{Q}})$ where $\gamma\neq 0$ because the curve \eqref{eq:standard_abc_family} becomes singular when $\alpha\beta\gamma=0$. This proves the following statement.

\begin{proposition}
Let $(\alpha,\beta,\gamma)\in \{(x,y,z):xyz\neq 0\}\subset \mathbb{P}^{2}(\overline{\mathbb{Q}})$ be a closed point. The curve $E_{(\alpha,\beta,\gamma)}$ is smooth and is isomorphic to $E_{(\alpha/\gamma,\beta/\gamma,1)}$. We also have the group isomorphism
\[E_{(\alpha,\beta,\gamma)}(\overline{\mathbb{Q}}(t))\cong E_{(\alpha/\gamma,\beta/\gamma,1)}(\overline{\mathbb{Q}}(t))\]
\end{proposition}
\noindent
The discriminant of the Weierstrass equation \eqref{eq:standard_abc_family} is defined by
\begin{equation}\label{eq:general_discriminant}
\Delta_{(\alpha,\beta,\gamma)}=\Delta(E_{(\alpha,\beta,\gamma)})(t)=16\cdot \Delta_{1}^{2}\cdot \Delta_{2}^{2}\cdot \Delta_{3}^{2}
\end{equation}
where
\begin{align*}
\Delta_{1}&=\alpha -2 \alpha  t^2-4 \beta  t^2+\alpha  t^4\\
\Delta_{2}&=\alpha -\gamma -2 \alpha  t^2-2 \gamma  t^2+\alpha  t^4-\gamma  t^4\\
\Delta_{3}&=\gamma -4 \beta  t^2+2 \gamma  t^2+\gamma  t^4
\end{align*}
The $j$-invariant of the family is equal to
\begin{equation}\label{eq:gen_j_invariant}
j_{(\alpha,\beta,\gamma)}=j(E_{(\alpha,\beta,\gamma)})(t) =\frac{j_{\textrm{num}}}{\Delta_{1}^2\cdot \Delta_{2}^2\cdot\Delta_{3}^2}
\end{equation}
where
\begin{equation*}
\begin{split}
j_{\textrm{num}}=2^8 \left(16 \beta ^2 t^4-4 \beta  \gamma  \left(t^3+t\right)^2+\alpha ^2 \left(t^2-1\right)^4-\right.\\\left.-\alpha  \left(t^2-1\right)^2 \left(4 \beta  t^2+\gamma  \left(t^2+1\right)^2\right)+\gamma ^2 \left(t^2+1\right)^4\right)^3
\end{split}
\end{equation*}
We observe that $\Delta_{(\alpha,\beta,\gamma)}(t)$ is a square of certain polynomial $\delta(t)$ in $\overline{\mathbb{Q}}[t]$. This polynomial $\delta(t)$ is separable in $\overline{\mathbb{Q}}[t]$ if and only if 
\begin{equation}\label{eq:generic_case}
\alpha  \beta  \gamma  (\alpha +\beta ) (\alpha -\gamma ) (\beta -\gamma ) (\alpha  \beta -\alpha  \gamma -\beta  \gamma ) \neq 0
\end{equation}

\begin{proposition}
Under the assumption \eqref{eq:generic_case} the curve $E_{(\alpha,\beta,\gamma)}$ is a globally minimal Weierstrass model of an elliptic curve. Its associated elliptic surface is an elliptic K3 surface and has bad fibers of type $I_{2}$ for $t$ such that $\Delta_{(\alpha,\beta,\gamma)}(t)=0$. It is smooth for $t=\infty$.
\end{proposition}
\begin{proof}
We deduce that equation \eqref{eq:standard_abc_family} is minimal at every finite place and has good reduction at infinity (\cite[VII, Remark 1.1]{Silverman_arithmetic}). The reduction types are $I_{2}$ for $t\neq \infty$ and such that $\Delta_{(\alpha,\beta,\gamma)}(t)=0$. Let $S$ be the elliptic surface for which $E_{(\alpha,\beta,\gamma)}$ is the generic fiber. By the results of Oguiso \cite[Theorem 1]{Oguiso_c2} and Shioda \cite[Theorem 2.8]{Shioda_Mordell_Weil} the Euler characteristic $\chi(S)=\chi(S,\mathcal{O}_{S})$ is equal to $2$. This implies that our surface $S$ is a K3 surface.
\end{proof}
If the condition \eqref{eq:generic_case} is violated, then our curve $E_{(\alpha,\beta,\gamma)}$ will be either singular (only for $\alpha\beta\gamma=0$) or will have a different configuration of singular fibers of the associated elliptic surface. The Zariski closed set 
\[V((\alpha +\beta ) (\alpha -\gamma ) (\beta -\gamma ) (\alpha  \beta -\alpha  \gamma -\beta  \gamma ))\subset\mathbb{P}^{2}\]
is the sum of irreducible components
\begin{align*}
V_{1,1,0}&=V(\alpha+\beta)\\
V_{1,0,-1}&=V(\alpha-\gamma)\\
V_{0,1,-1}&=V(\beta-\gamma)\\
V_{q}&=V(\alpha \beta - \alpha \gamma - \beta \gamma)
\end{align*}
Let us denote by $U$ the open set $\mathbb{P}^{2}\setminus V(\alpha\beta\gamma)$. We say that a triple $(\alpha,\beta,\gamma)\in U$ is \emph{generic} if $(\alpha,\beta,\gamma)$ does not belong to any of the closed sets $V_{1,1,0},V_{1,0,-1},V_{0,1,-1}, V_{q}$. The set of such triples is again Zariski open. We denote it by $\mathcal{U}_{\textrm{gen}}$.

\subsection{Generic triple $(\alpha,\beta,\gamma)$}\label{subsec:points_general}
In the generic case an elliptic surface $\mathcal{E}_{(\alpha,\beta,\gamma)}$ attached to $E_{(\alpha,\beta,\gamma)}$ will be a K3 surface hence its Picard rank satisfies the inequality $\rho(\mathcal{E}_{(\alpha,\beta,\gamma)})\leq 20$. Application of Shioda-Tate formula shows that the rank of $E_{(\alpha,\beta,\gamma)}(\overline{\mathbb{Q}}(t))$ is at most $6$ in this case. Let us consider $6$ points on this curve. To simplify the notation assume for now that $A=(t^2-1)^2$, $B=4t^2$ and  $q(\alpha,\beta,\gamma)=\alpha\beta -\alpha\gamma - \beta\gamma$.
\begin{align*}
P_{1}:\ x(P_{1})=0,\quad & y^2(P_{1})=\alpha\beta\gamma AB(A+B)\\
P_{2}:\ x(P_{2})=-\alpha A+(q(\alpha,\beta,\gamma)/\gamma)B,\quad &  y^{2}(P_{2})=\frac{\alpha B (\gamma-\alpha) q(\alpha,\beta,\gamma) (A \gamma-\beta B+B \gamma)^2}{\gamma^3}\\
P_{3}:\ x(P_{3})=\beta\gamma/(-\beta + \gamma) A,\quad &  y^{2}(P_{3})=\frac{A \beta \gamma q(\alpha,\beta,\gamma) (A \gamma-\beta B+B \gamma)^2}{(\beta-\gamma)^3}\\
P_{4}:\ x(P_{4})=-(\alpha\beta)/(\alpha + \beta) (A + B),\quad &  y^{2}(P_{4})=\frac{\alpha \beta (A+B) q(\alpha,\beta,\gamma) (\alpha A-\beta B)^2}{(\alpha+\beta)^3}\\
P_{5}:\ x(P_{5})=-\alpha(A+B), \quad&  y^{2}(P_{5})=-\alpha B (\alpha-\gamma) (A+B) (\alpha A+\alpha B-\beta B)\\
P_{6}:\ x(P_{6})=\beta A,\quad &  y^{2}(P_{6})=A \beta (\alpha+\beta) (A+B) (A \beta+A \gamma+B \gamma)
\end{align*}
Without further assumptions the points $P_{1},P_{2},P_{3}$ and $P_{4}$ all belong to $E_{(\alpha,\beta,\gamma)}(\overline{\mathbb{Q}}(t))$. Point $P_{5}$ with such a choice of $x$-coordinate exists if and only if $\alpha=\beta$. The point $P_{6}$ is well-defined on $E_{(\alpha,\beta,\gamma)}$ only when $\beta=-\gamma$. We denote by $\mathcal{G}$ the set of generic triples $(\alpha,\beta,\gamma)$ such that $\alpha\neq\beta$ and $\beta\neq-\gamma$.

\begin{lemma}\label{lem:rank_4_generic}
Let $(\alpha,\beta,\gamma)\in\mathcal{G}$ be a generic triple. The set $\{P_{1},P_{2},P_{3},P_{4}\}$ spans a rank $4$ subgroup of $E_{(\alpha,\beta,\gamma)}(\overline{\mathbb{Q}}(t))$.
Their height pairing matrix $\left(\langle P_{i},P_{j}\rangle\right)_{1\leq i,j\leq 4}$ looks as follows
\begin{equation*}
\left(
\begin{array}{cccc}
4 & 0 & 0 & 0\\
0 & 2 & 0 & 0\\
0 & 0 & 2 & 0\\
0 & 0 & 0 & 2
\end{array}\right)
\end{equation*}
\end{lemma}
\begin{proof}
We denote by $S$ the Kodaira-N\'{e}ron model of $E_{(\alpha,\beta,\gamma)}$.
It has bad fibers over points $t_{0}\in\overline{\mathbb{Q}}$ such that $(\alpha A-\beta B)(t_{0})=0$ or $(\alpha A-\gamma(A+B))(t_{0})=0$ or $(\beta B-\gamma(A+B))(t_{0})=0$.
All bad fibers are of type $I_{2}$ and the image $\overline{P_{i}}$ of the section $P_{i}:\mathbb{P}^{1}\rightarrow S$ over the point $t_{0}$ lies in the component which does not intersect the component of the zero section $O:\mathbb{P}^{1}\rightarrow S$ if and only if $y^{2}(P_{i})(t_{0})=0$ and the first coordinate reduces to the first coordinate of the appropriate two-torsion point which happens to be singular at that fiber.
The height pairing is symmetric so we have to compute only the values $\langle P_{i},P_{j}\rangle$ for $i\leq j$.

From the equation of $E_{(\alpha,\beta,\gamma)}$ it follows that $\pi:S\rightarrow\mathbb{P}^{1}$ is a K3 surface, hence $\chi(S)=2$. Moreover, we check that for $i=1,2,3,4$ we have $\overline{P_{i}}.\overline{O}=0$. This is easy to see for all $t\neq\infty$. For $t=\infty$ we make a change of coordinates $t=1/s$ and $(x,y)\mapsto (x s^4,ys^6)$ and look at the fiber at $s=0$. Now observe that the correcting terms $c_{v}(P_{1})$ for $P_{1}$ are all zero, hence $\langle P_{1},P_{1}\rangle = 4$. For $\langle P_{1},P_{2}\rangle$ we have that $c_{v}(P_{1},P_{2})=0$ for all points $v$ and we check that $\overline{P_{1}}.\overline{P_{2}}=2$. This is true because we have a system of equations
\begin{equation}\label{eq:system_eq}
\begin{array}{rl}
x(P_{1})&=\ x(P_{2})\\
y^2(P_{1})&=\ y^2(P_{2})
\end{array}
\end{equation}
and the number of its solutions equals the intersection number $(\overline{P_{1}}+\overline{-P_{1}}).(\overline{P_{2}}+\overline{-P_{2}})$.
The elements $t_{0}$ that satisfy the system \eqref{eq:system_eq} are the one that satisfy 
\[(-\alpha \beta B + \alpha\gamma A + \alpha\gamma B + \beta\gamma B)(t_{0})=0.\]
Defining polynomial is separable for $(\alpha,\beta,\gamma)\in\mathcal{G}$ and we can easily check that $\pi^{-1}(t_{0})$ is never a singular fiber. For $t_{0}=\infty$ by a change of coordinates we easily check that there is no solution for $s=0$. Hence
\[(\overline{P_{1}}+\overline{-P_{1}}).(\overline{P_{2}}+\overline{-P_{2}})=4\]
The involution $\iota:P\mapsto -P$ on the generic fiber $E_{(\alpha,\beta,\gamma)}$ extends to an isomorphism on $S$ and it preserves the intersection numbers. The divisor $\overline{P_{2}}+\overline{-P_{2}}$ is invariant under $\iota$ and $\iota(\overline{P_{1}})=\overline{-P_{1}}$. This implies that
\[\overline{P_{1}}.(\overline{P_{2}}+\overline{-P_{2}})=2.\]
Now we observe that $\overline{P_{1}}$ cannot intersect both $\overline{P_{2}}$ and $\overline{-P_{2}}$ for our choice of $t_{0}$. Without loss of generality we can choose square roots in such a way that $\overline{P_{1}}.\overline{P_{2}}=2$. This implies that $\langle P_{1},P_{2}\rangle =0$. 

A similar computation shows that $\langle P_{1},P_{3}\rangle =0$. In this case what we have to use is the fact that $A$ is not separable and both solutions $t=\pm 1$ count twice. The same way we obtain $\langle P_{1},P_{4}\rangle =0$. 

We claim that $\langle P_{2},P_{2}\rangle=2$. This is easily checked because for $t_{0}$ such that $(\gamma (A+B)-\beta B)(t_{0})=0$ the point $P_{2}(t_{0})$ is singular on the fiber above $t_{0}$ and in the blow-up it is moved to the other component of the $I_{2}$ fiber above $t_{0}$, so the correcting terms for $v=t_{0}$ are $c_{v}(P_{2})=1/2$. Polynomial $(\gamma (A+B)-\beta B)(t)$ is separable, hence the claim follows by height formula \eqref{equation:height_formula}.

Now we prove that $\langle P_{2},P_{3}\rangle=0$. The common intersection would appear for $t_{0}$ such that $(\gamma(A+B)-\beta B)(t_{0})$, hence in the fibers of bad reduction. This implies already that $c_{v}(P_{2},P_{3})=1/2$ for $v=t_{0}$. To prove that $\overline{P_{2}}.\overline{P_{3}}=0$ is equivalent to proving that $\overline{P_{2}-P_{3}}.\overline{O}=0$. To check this we use the addition formula and $x(P_{2}+P_{3})$ is a degree $4$ polynomial in $t$ with nonzero free coefficient
\[-\frac{\alpha  \gamma ^2}{-\alpha  \beta +\alpha  \gamma +\beta  \gamma },\]
hence the divisor $\overline{P_{2}-P_{3}}$ never intersects the divisor $\overline{O}$ at any place.

We claim that $\langle P_{2},P_{4}\rangle=0$. The solutions of the system $x(P_{2})=x(P_{4}),\ y^2(P_{2})=y^2(P_{4})$ in $t$ lead to $t_{0}$ that satisfy $(\alpha ^2 A \gamma +\alpha ^2 \beta  (-B)+\alpha ^2 B \gamma -\alpha  \beta ^2 B+\alpha  \beta  B \gamma +\beta ^2 B \gamma)(t_{0})=0$. These points $t_{0}$ do not coincide with the places of bad reduction because of the assumptions on $(\alpha,\beta,\gamma)$, hence $\overline{P_{2}}.\overline{P_{4}}=2$. Now the correcting terms $c_{v}(P_{2},P_{4})$ are always zero because the points never meet the same component at the fibers of bad reduction. This proves the claim.

To prove $\langle P_{3},P_{3}\rangle = 2$ we proceed as in the case of point $P_{2}$. Next we show that $\langle P_{3},P_{4}\rangle = 0$. The solutions $t_{0}$ to the system $x(P_{3})=x(P_{4}),\ y^2(P_{1})=y^2(P_{2})$ satisfy $(\alpha  A \beta -2 \alpha  A \gamma -A \beta  \gamma +\alpha  \beta  B-\alpha  B \gamma)(t_{0})=0$ and again they never meet the points for which we have bad reduction, hence $\overline{P_{2}}.\overline{P_{4}}=2$ and the correcting terms $c_{v}(P_{3},P_{4})$ are zero for all $v$.

Now to finish the proof we show that $\langle P_{4},P_{4}\rangle = 2$ but this is proved the same way as for $P_{2}$ and $P_{3}$. We conlude by saying that the height pairing matrix of points $P_{i}$, $i=1,2,3,4$ has nonzero determinant, so the points are linearly independent.
\end{proof}

\begin{proposition}
There exists a triple $(\alpha,\beta,\gamma)\in\mathcal{G}$ such that
\[\textrm{rank}\ E_{(\alpha,\beta,\gamma)}(\overline{\mathbb{Q}}(t))=4.\]
\end{proposition}
\begin{proof}
We put $\alpha=3,\beta=5$ and $\gamma=1$. We have to show that $\rho(\mathcal{E}_{(\alpha,\beta,\gamma)})\leq 18$. The elliptic surface $\mathcal{E}_{(3,5,1)}$ associated with $E_{(3,5,1)}$ has good reduction at $p=1009$ by Lemma \ref{lem:good_reduction} and the N\'{e}ron-Severi group is defined over $\mathbb{F}_{p}$. We compute the characteristic polynomial of Frobenius automorphism $\Phi$ for any prime $\ell\neq p$. It follows that
\[P(x)=(x-1009)^{18}(x^4 + 412x^3 - 801146x^2 + 419449372x + 1009^4).\]
The last factor is irreducible over $\mathbb{Z}[t]$ and is not a cyclotomic polynomial, hence $R_{\Phi}=18$. Corollary \ref{cor:spec_NS} implies that $\rho(\mathcal{E}_{(3,5,1)})\leq 18$. Now application of Shioda-Tate formula finishes the proof.
\end{proof}

\begin{lemma}\label{lem:rank_5_P6}
Let $(\alpha,\beta,\gamma)\in\mathcal{U}_{\textrm{gen}}$ be a generic triple such that $\beta=-\gamma$ and $\alpha\neq\beta$. The set $\{P_{1},P_{2},P_{3},P_{4},P_{6}\}$ spans a rank $5$ subgroup of $E_{(\alpha,\beta,\gamma)}(\overline{\mathbb{Q}}(t))$. Its height pairing matrix looks as follows
\begin{equation*}
\left(
\begin{array}{ccccc}
4 & 0 & 0 & 0 & 0\\
0 & 2 & 0 & 0 & 0\\
0 & 0 & 2 & 0 & 0\\
0 & 0 & 0 & 2 & 0\\
0 & 0 & 0 & 0 & 4\\
\end{array}\right)
\end{equation*}
\end{lemma}
\begin{proof}
By Lemma \ref{lem:rank_4_generic} we already know that the points $P_{1},P_{2},P_{3}$ and $P_{4}$ span a rank $4$ subgroup of the full Mordell-Weil group. We have to compute now the intersections $\langle P_{i},P_{6}\rangle$ for $i=1,2,3,4$ and $\langle P_{6},P_{6}\rangle$.

To show that $\langle P_{6},P_{6}\rangle=4$ we check that $\overline{P_{6}}.\overline{O}=0$ and that $c_{v}(P_{6})=0$ for all $v$. Now observe that $\langle P_{i},P_{6}\rangle = 2-\overline{P_{i}}.\overline{P_{6}}$ for $i=1,2,3,4$. We will prove that $\overline{P_{i}}.\overline{P_{6}}=2$. This is equivalent to prove that the system $x(P_{i})=x(P_{6}),\ y^2(P_{i})=y^2(P_{6})$ has exactly four solutions (some of them possibly multiple). For $i=1,3$ the solutions come from the equation $A(t)=0$. For $i=2$, we obtain the solutions from $(\alpha  A-A \gamma +2 \alpha  B-B \gamma)(t)=0$. For $i=4$ we get $(-2 \alpha  A+A \gamma -\alpha  B)(t)=0$.
\end{proof}

\begin{proposition}\label{prop:rank_5_example}
There exists a triple $(\alpha,\beta,\gamma)\in\mathcal{U}_{\textrm{gen}}$ such that $\beta=-\gamma$ and $\alpha\neq\beta$ and
\[\textrm{rank}\ E_{(\alpha,\beta,\gamma)}(\overline{\mathbb{Q}}(t))=5.\]
\end{proposition}
\begin{proof}
We put $\alpha=3,\beta=-1$ and $\gamma=1$. We check that the surface $S=\mathcal{E}_{(3,-1,1)}$ has good reduction at primes $p=241,409$ with the full N\'{e}ron-Severi group of $S_{\overline{\mathbb{F}}_{p}}$ defined over $\mathbb{F}_{p}$. We compute the characteristic polynomials $P_{p}$ of $\Phi_{p}$ for a fixed $\ell\neq p$
\begin{align*}
P_{241}(x)&=(x-241)^{20}(x^2+478x+241^2)\\
P_{409}(x)&=(x-409)^{20}(x^2-626x+409^2)\\
\end{align*}
Now by Corollary \ref{cor:disc_mod_p} we obtain
\begin{align*}
\Delta(\NS(S_{\overline{\mathbb{F}}_{241}}))&\equiv -3\cdot 5\textrm{  }\textrm{mod }(\mathbb{Q}^{\times})^2,\\
\Delta(\NS(S_{\overline{\mathbb{F}}_{409}}))&\equiv -3\textrm{  }\textrm{mod }(\mathbb{Q}^{\times})^2.
\end{align*}
Suppose that the rank $E_{(-3,-1,1)}(\overline{\mathbb{Q}}(t))$ would be equal to $6$. Then by Shioda-Tate formula this will imply that $N=\rho(S_{\overline{\mathbb{Q}}})=20$. The image $\spec_{p}(N)$ would be a finite index subgroup in the codomain. So this will imply that the discriminants $\Delta(\NS(S_{\overline{\mathbb{F}}_{409}}))$ and $\Delta(\NS(S_{\overline{\mathbb{F}}_{241}}))$ should be equal modulo squares. But they are not, hence a contradiction. This implies that $\rho(N)\leq 19$ and by Shioda-Tate formula again we obtain the statement of the proposition.
\end{proof}
\begin{lemma}
Let $(\alpha,\beta,\gamma)\in\mathcal{U}_{\textrm{gen}}$ be a generic triple such that $\alpha=\beta$ and $\beta\neq-\gamma$. The set $\{P_{1},P_{2},P_{3},P_{4},P_{5}\}$ spans a rank $5$ subgroup of $E_{(\alpha,\beta,\gamma)}(\overline{\mathbb{Q}}(t))$. Their height pairing matrix looks as follows
\begin{equation*}
\left(
\begin{array}{ccccc}
4 & 0 & 0 & 0 & 0\\
0 & 2 & 0 & 0 & 0\\
0 & 0 & 2 & 0 & 0\\
0 & 0 & 0 & 2 & 0\\
0 & 0 & 0 & 0 & 4\\
\end{array}\right)
\end{equation*}
\end{lemma}
\begin{proof}
The proof is similar to the proof of Lemma \ref{lem:rank_5_P6}.
\end{proof}

\begin{proposition}
There exists a triple $(\alpha,\beta,\gamma)\in\mathcal{U}_{\textrm{gen}}$ such that $\alpha=\beta$ and $\beta\neq-\gamma$ and
\[\textrm{rank}\ E_{(\alpha,\beta,\gamma)}(\overline{\mathbb{Q}}(t))=5.\]
\end{proposition}
\begin{proof}
We put $\alpha=3,\beta=3$ and $\gamma=1$ and consider $S=\mathcal{E}_{(3,3,1)}$. We check that $S$ has good reduction at primes $p=73,97$. Like in the proof of Proposition \ref{prop:rank_5_example} we compute the discriminant of the reduction of N\'{e}ron-Severi groups. The characteristic polynomials of Frobenius are as follows
\begin{align*}
P_{73}(x)&=(x-73)^{20}(x^2+142x+73^2)\\
P_{97}(x)&=(x-97)^{20}(x^2-2x+97^2)
\end{align*}
We obtain $\modulo{\Delta(\NS(S_{\overline{\mathbb{F}}_{73}}))}{-2}{(\mathbb{Q}^{\times})^2}$ and $\modulo{\Delta(\NS(S_{\overline{\mathbb{F}}_{97}}))}{-3}{(\mathbb{Q}^{\times})^2}$. This leads to the conclusion that $\rho(S)\leq 19$ which implies the statement of the proposition.
\end{proof}
\subsection{Type $(\alpha,\beta,\gamma)=(-1,-1,1)$}
Now we consider a very special generic triple $(\alpha,\beta,\gamma)$. It satisfies two extra conditions $\alpha=\beta$ and $\beta=-\gamma$. By the projective equivalence of tuples we can state the results for $\alpha=\beta=-\gamma=-1$.
The Kodaira types of singular fibers that appear in this case are:
\begin{equation}
\begin{array}{rl}
\textrm{point} & \textrm{fiber type}\\
\hline
 t^2 + 2t - 1=0 & I_{2}\\
 t^4+1=0 & I_{2}\\
 t^2 - 2t - 1=0 & I_{2}\\
 t^4 + 6t^2 + 1=0 & I_{2}
\end{array}
\end{equation}
There are six linearly independent points on the curve:
\begin{align*}
R_{1}&=(0,fgh)\\
R_{2}&=(h^2,\sqrt{2}fgh)\\
R_{3}&=(1/2 h^2,\sqrt{-6}(h^2-2g^2)h)\\
R_{4}&=(-f^2,\sqrt{2}fgh)\\
R_{5}&=(-1/2 g^2, \sqrt{6}g(2h^2-g^2))\\
R_{6}&=(-1/2 f^2,\sqrt{6}f(2h^2-f^2))
\end{align*}
With respect to the height pairing $\langle\cdot,\cdot\rangle$ the determinant of Gram matrix equals $384$
\begin{equation}
\left(\begin{array}{cccccc}
 4 &  0 &  0 &  0 &  0 &  0\\                                                                                                                                                                                         
 0 &  4 &  0 & -2 &  0 &  0\\                                                                                                                                                                                         
 0 &  0 &  2 &  0 &  0 &  0\\                                                                                                                                                                                         
 0 & -2 &  0 &  4 &  0 &  0\\
 0 &  0 &  0 &  0 &  2 &  0\\
 0 &  0 &  0 &  0 &  0 &  2
\end{array}\right)
\end{equation}
By the Shioda-Tate formula this proves that we have found sufficiently many independent points to obtain at least a finite index subgroup in the full Mordell-Weil group. In fact we can easily determine the full Mordell-Weil group using a full descent computation like in \cite[Lemma 6.2]{Naskrecki_Acta}. We denote by $T_{1}$ and $T_{2}$ the generators of the torsion subgroup. By the fibers configuration it is easy to prove (cf. \cite[Cor. 7.5]{Shioda_Schutt}) that the full torsion subgroup is isomorphic to $\mathbb{Z}/2\oplus\mathbb{Z}/2$. We put $T_{1}=(4t^2,0)$ and $T_{2}=(-t^4-2t^2-1,0)$. We consider $6$ points $Q_{1},Q_{2},Q_{3},Q_{4},Q_{5},Q_{6}$ with coefficients in the field $\mathbb{Q}(\sqrt{2},\sqrt{3},\sqrt{-3})(t)$ defined as follows
\begin{align}
2Q_{1}&=R_1 - R_2 - R_4 - R_5 + R_6 - T_2,\\
2Q_{2}&=R_1 - R_2 - R_5 + R_6 - T_2,\\
2Q_{3}&=R_1 - R_4 - R_5 + R_6 - T_2,\\
2Q_{4}&=R_1 - R_2 + R_3 - R_4 - R_5 + T_1 - T_2,\\
2Q_{5}&=R_1 - R_2 - R_3 - R_4 - R_5 + T_1 - T_2,\\
2Q_{6}&=R_1 - R_2 + R_3 - R_4 + R_6 + T_1. 
\end{align}
The height pairing matrix for the points $Q_{i}$ has the form
\begin{equation}
\left(\begin{array}{cccccc}
3 & 5/2 & 5/2 & 5/2 & 5/2 & 5/2\\
5/2 & 3 & 3/2 & 2 & 2 & 2\\
5/2 & 3/2 & 3 & 2 & 2 & 2\\
5/2 & 2 & 2 & 3 & 2 & 5/2\\
5/2 & 2 & 2 & 2 & 3 & 3/2\\
5/2 & 2 & 2 & 5/2 & 3/2 & 3\\
\end{array}\right)
\end{equation}
and determinant equal to $3/8$. Now we follow the approach in proof of \cite[Lemma 6.2]{Naskrecki_Acta} and prove that the points $Q_{i}$ and torsion generators $T_{1}$ and $T_{2}$ span the whole Mordell-Weil group over $\overline{\mathbb{Q}}(t)$. We omit tedious computations which can be easily done by a computer package. This allows us to deduce that the discriminant of the N\'{e}ron-Severi group in this case will be $-2^5\cdot 3$. This can be used to deduce the supersingular primes $p$ that allow in positive characteristic to get rank $8$ over $\mathbb{F}_{p^2}$, cf. Section \ref{sec:supersingular}.

The number field $F=\mathbb{Q}(\sqrt{2},\sqrt{3},\sqrt{-3})$ is normal and its Galois group is isomorphic to $(\mathbb{Z}/2)^3$. Generators $\tau_{1},\tau_{2}, \tau_{3}$ are determined easily
\begin{align*}
\tau_{1}(\sqrt{2})=\sqrt{2},\quad\tau_{1}(\sqrt{3})=\sqrt{3},\quad\tau_{1}(\sqrt{-3})=-\sqrt{-3}\\
\tau_{2}(\sqrt{2})=-\sqrt{2},\quad\tau_{2}(\sqrt{3})=-\sqrt{3},\quad\tau_{2}(\sqrt{-3})=\sqrt{-3}\\
\tau_{3}(\sqrt{2})=-\sqrt{2},\quad\tau_{3}(\sqrt{3})=\sqrt{3},\quad\tau_{3}(\sqrt{-3})=-\sqrt{-3}
\end{align*}

Group $\textrm{Gal}(F/\mathbb{Q})$ acts naturally on the group $G=E_{(-1,-1,1)}(\overline{\mathbb{Q}}(t))=E_{(-1,-1,1)}(F(t))$ (cf. \cite[Wn. 2.5.18]{Naskrecki_thesis}) and we obtain a Galois action on the free module $G/G_{\textrm{tors}}$ which determines a representation
\[\rho:\textrm{Gal}(F/\mathbb{Q})\rightarrow\textrm{GL}_{6}(\mathbb{Z}).\]
We represent the matrices in the basis $\{Q_{i}+G_{\textrm{tors}}\}$ of $G/G_{\textrm{tors}}$
\[\rho(\tau_{1})=\left(\begin{array}{cccccc}
1 & 0 & 0 & 0 & 0 & 0\\
0 & 1 & 0 & 0 & 0 & 0\\
0 & 0 & 1 & 0 & 0 & 0\\
0 & 0 & 0 & 0 & 1 & -1\\
0 & 0 & 0 & 1 & 0 & 1\\
0 & 0 & 0 & 0 & 0 & 1
\end{array}\right)
\]

\[\rho(\tau_{2})=\left(\begin{array}{cccccc}
-3 & -2 & -2 & -4 & -4 & -4\\
2 & 1 & 2 & 2 & 2 & 2\\
2 & 2 & 1 & 2 & 2 & 2\\
0 & 0 & 0 & 0 & 1 & -1\\
0 & 0 & 0 & 1 & 0 & 1\\
0 & 0 & 0 & 0 & 0 & 1
\end{array}\right)
\]
\[\rho(\tau_{3})=\left(\begin{array}{cccccc}
-7 & -6 & -6 & -6 & -6 & -6\\
2 & 1 & 2 & 2 & 2 & 2\\
2 & 2 & 1 & 2 & 2 & 2\\
0 & 0 & 0 & 0 & -1 & 1\\
2 & 2 & 2 & 1 & 2 & 1\\
2 & 2 & 2 & 2 & 2 & 1\\
\end{array}\right)
\]

The tuples $v\in\mathbb{Z}^{6}$ such that $\rho(\sigma)v=v$ for all $\sigma\in\textrm{Gal}(F/\mathbb{Q})$ represent the points $P$ in $E_{(-1,-1,1)}(F(t))$ such that $\sigma(P)-P\in E[2](F(t))$ for all $\sigma$. We easily find that the submodule $\{v: \rho(\tau_{1}\tau_{3})v=v\}$ is one-dimensional and is generated by the vector $(3,-1,-1,0,-1,-1)^{T}$. This combination represents a point $Q=3Q_{1}-Q_{2}-Q_{3}-Q_{5}-Q_{6}$ which satisfies $\tau_{i}(Q)=Q+T_{2}$ for $i=1,2,3$. This implies that the free part of $E_{(-1,-1,1)}(\mathbb{Q}(t))$ is spanned by $2Q=P_{1}$, hence
\[E_{(-1,-1,1)}(\mathbb{Q}(t))\cong\mathbb{Z}\oplus\mathbb{Z}/2\oplus\mathbb{Z}/2.\]

\subsection{Triples $(\alpha,\beta,\gamma)\in U$ with $\alpha=\gamma$}
Let us denote the set of triples $(\alpha,\beta,\gamma)\in U$ with $\alpha=\gamma$ and $\alpha\neq -\beta$ and $\beta\neq\gamma$ by $\mathcal{S}_{1}$. In this situation the equation $E_{(\alpha,\beta,\gamma)}$ defines an elliptic curve which is a generic fiber of elliptic surface $\mathcal{E}_{(\alpha,\beta,\gamma)}$ with bad fibers of types $I_{2}$ and $I_{4}$. The discriminant of the Weierstrass equation of $E_{(\alpha,\beta,\gamma)}$ has the form
\[256 \alpha ^2 t^4 \left(\alpha +\alpha  t^4-2 \alpha  t^2-4 \beta  t^2\right)^2 \left(\alpha +\alpha  t^4+2 \alpha  t^2-4 \beta  t^2\right)^2\]
For $t=0$ and $t=\infty$ we have the $I_{4}$ fibers and for $t$ that are roots of 
\[\left(\alpha +\alpha  t^4-2 \alpha  t^2-4 \beta  t^2\right) \left(\alpha +\alpha  t^4+2 \alpha  t^2-4 \beta  t^2\right)=0\]
we have $I_{2}$ fibers. Shioda-Tate formula implies that we have the bound
\[\rank E_{(\alpha,\beta,\gamma)}(\overline{\mathbb{Q}}(t))\leq 4.\]

\begin{lemma}\label{lem:S1_case}
Let $(\alpha,\beta,\gamma)\in\mathcal{S}_{1}$. The set of points $\{P_{1},P_{3},P_{4}\}$ spans a rank $3$ subgroup of $E_{(\alpha,\beta,\gamma)}(\overline{\mathbb{Q}}(t))$. Their height pairing matrix looks as follows
\begin{equation*}
\left(
\begin{array}{ccccc}
4 & 0 & 0\\
0 & 2 & 0\\
0 & 0 & 2
\end{array}\right)
\end{equation*}
\end{lemma}
\begin{proof}
We observe first that the points $P_{1},P_{3}$ and $P_{4}$ are the only well-defined point from the list in \S\ref{subsec:points_general} that are not two-torsion. Essentially we reprove part of Lemma \ref{lem:rank_4_generic} with the extra assumption $\alpha=\gamma$. Observe that the assumption does not change the fiber types of bad reduction crucial for our computation.
\end{proof}

\begin{proposition}
There exists a triple $(\alpha,\beta,\gamma)\in\mathcal{S}_{1}$ and
\[\textrm{rank}\ E_{(\alpha,\beta,\gamma)}(\overline{\mathbb{Q}}(t))=3.\]
\end{proposition}
\begin{proof}
We put $\alpha=1$, $\beta=4$ and $\gamma=1$ and we consider $S=\mathcal{E}_{(1,4,1)}$. The surface $S$ has good reduction at primes $p=61$ and $p=181$ with its N\'{e}ron-Severi group defined already over $\mathbb{F}_{p}$. The characteristic polynomials of Frobenius are as follows
\begin{align*}
P_{61}(x)&=(x-61)^{20}(x^2+118x+61^2)\\
P_{181}(x)&=(x-181)^{20}(x^2+166x+181^2)\\
\end{align*}
We obtain $\modulo{\Delta(\NS(S_{\overline{\mathbb{F}}_{61}}))}{-3\cdot 5}{(\mathbb{Q}^{\times})^2}$ and $\modulo{\Delta(\NS(S_{\overline{\mathbb{F}}_{97}}))}{-3\cdot 11}{(\mathbb{Q}^{\times})^2}$. This leads to the conclusion that $\rho(S)\leq 19$ which implies the statement of the proposition.
\end{proof}

\subsection{Triples $(\alpha,\beta,\gamma)\in U$ with $\alpha=-\beta$}
Let us denote the set of triples $(\alpha,\beta,\gamma)\in U$ with $\alpha=-\beta$ and $\alpha\neq \gamma$ and $\beta\neq\gamma$ by $\mathcal{S}_{2}$. In this situation the equation $E_{(\alpha,\beta,\gamma)}$ defines an elliptic curve which is a generic fiber of elliptic surface $\mathcal{E}_{(\alpha,\beta,\gamma)}$ with bad fibers of types $I_{2}$ and $I_{4}$. The discriminant of the Weierstrass equation of $E$ has the form
\[16 \alpha ^2 \left(t^2+1\right)^4 \left(\alpha -\gamma +\alpha  t^4-\gamma  t^4-2 \alpha  t^2-2 \gamma  t^2\right)^2 \left(\gamma +\gamma  t^4+4 \alpha  t^2+2 \gamma  t^2\right)^2\]
For each root of $t^2+1$ we have an $I_{4}$ fiber and for $t$ that are roots of the remaining factors of the discriminant
we have $I_{2}$ fibers. Shioda-Tate formula implies that we have the bound
\[\rank E_{(\alpha,\beta,\gamma)}(\overline{\mathbb{Q}}(t))\leq 4.\]

\begin{lemma}
Let $(\alpha,\beta,\gamma)\in\mathcal{S}_{2}$. The set of points $\{P_{1},P_{2},P_{3}\}$ spans a rank $3$ subgroup of $E_{(\alpha,\beta,\gamma)}(\overline{\mathbb{Q}}(t))$. Their height pairing matrix looks as follows
\begin{equation*}
\left(
\begin{array}{ccccc}
4 & 0 & 0\\
0 & 2 & 0\\
0 & 0 & 2
\end{array}\right)
\end{equation*}
\end{lemma}
\begin{proof}
We mimic the proof of Lemma \ref{lem:S1_case}.
\end{proof}

\begin{proposition}
There exists a triple $(\alpha,\beta,\gamma)\in\mathcal{S}_{2}$ and
\[\textrm{rank}\ E_{(\alpha,\beta,\gamma)}(\overline{\mathbb{Q}}(t))=3.\]
\end{proposition}
\begin{proof}
We put $\alpha=5$, $\beta=-5$ and $\gamma=1$ and we consider $S=\mathcal{E}_{(5,-5,1)}$. The surface $S$ has good reduction at primes $p=29$ and $p=101$ with its N\'{e}ron-Severi group defined already over $\mathbb{F}_{p}$. The characteristic polynomials of Frobenius are as follows
\begin{align*}
P_{29}(x)&=(x-29)^{20}(x^2+54x+29^2)\\
P_{101}(x)&=(x-101)^{20}(x^2-122x+101^2)\\
\end{align*}
We obtain $\modulo{\Delta(\NS(S_{\overline{\mathbb{F}}_{29}}))}{-7}{(\mathbb{Q}^{\times})^2}$ and $\modulo{\Delta(\NS(S_{\overline{\mathbb{F}}_{97}}))}{-5}{(\mathbb{Q}^{\times})^2}$. This leads to the conclusion that $\rho(S)\leq 19$ which implies the statement of the proposition.
\end{proof}

\subsection{Case $\beta=\gamma$}
In this section let us denote the set of triples $(\alpha,\beta,\gamma)\in U$ with $\beta=\gamma$ and $\alpha\neq \gamma$ and $\alpha\neq-\beta$ by $\mathcal{S}_{3}$. In this situation the equation $E_{(\alpha,\beta,\gamma)}$ defines an elliptic curve which is a generic fiber of elliptic surface $\mathcal{E}_{(\alpha,\beta,\gamma)}$ with bad fibers of types $I_{2}$ and $I_{4}$. The discriminant of the Weierstrass equation of $E$ has the form
\[16 \beta ^2 (t-1)^4 (t+1)^4 \left(\alpha +\alpha  t^4-2 \alpha  t^2-4 \beta  t^2\right)^2 \left(\alpha -\beta +\alpha  t^4-\beta  t^4-2 \alpha  t^2-2 \beta  t^2\right)^2\]
For $t=1$ and $t=-1$ we have an $I_{4}$ fiber and for $t$ that are roots of the remaining factors of the discriminant
we have $I_{2}$ fibers. Shioda-Tate formula implies that we have the bound
\[\rank E_{(\alpha,\beta,\gamma)}(\overline{\mathbb{Q}}(t))\leq 4.\]

\begin{lemma}
Let $(\alpha,\beta,\gamma)\in\mathcal{S}_{3}$. The set of points $\{P_{1},P_{2},P_{4}\}$ spans a rank $3$ subgroup of $E_{(\alpha,\beta,\gamma)}(\overline{\mathbb{Q}}(t))$. Their height pairing matrix looks as follows
\begin{equation*}
\left(
\begin{array}{ccccc}
4 & 0 & 0\\
0 & 2 & 0\\
0 & 0 & 2
\end{array}\right)
\end{equation*}
\end{lemma}
\begin{proof}
We mimic the proof of Lemma \ref{lem:S1_case}.
\end{proof}

\begin{proposition}
There exists a triple $(\alpha,\beta,\gamma)\in\mathcal{S}_{3}$ and
\[\textrm{rank}\ E_{(\alpha,\beta,\gamma)}(\overline{\mathbb{Q}}(t))=3.\]
\end{proposition}
\begin{proof}
We put $\alpha=5$, $\beta=1$ and $\gamma=1$ and we consider $S=\mathcal{E}_{(5,1,1)}$. The surface $S$ has good reduction at primes $p=29$ and $p=101$ with its N\'{e}ron-Severi group defined already over $\mathbb{F}_{p}$. The characteristic polynomials of Frobenius are as follows
\begin{align*}
P_{29}(x)&=(x-29)^{20}(x^2+22x+29^2)\\
P_{101}(x)&=(x-101)^{20}(x^2+106x+101^2)\\
\end{align*}
We obtain $\modulo{\Delta(\NS(S_{\overline{\mathbb{F}}_{29}}))}{-5}{(\mathbb{Q}^{\times})^2}$ and $\modulo{\Delta(\NS(S_{\overline{\mathbb{F}}_{97}}))}{-23}{(\mathbb{Q}^{\times})^2}$. This leads to the conclusion that $\rho(S)\leq 19$ which implies the statement of the proposition.
\end{proof}

\subsection{Further degeneration}
In this paragraph we list the remaining two special cases which are determined by triples $(\alpha,\beta,\gamma)$ in the intersections of $U\cap V_{1,0,-1}\cap V_{1,1,0}$ and  $U\cap V_{0,1,-1}\cap V_{1,1,0}$. 

\textbf{Type $(1,-1,-1)$}. The Kodaira types of singular fibers that appear in this case are:
\begin{equation}
\begin{array}{rl}
\textrm{point} & \textrm{fiber type}\\
\hline
 t^4+1=0 & I_{2}\\
 t=1 & I_{4}\\
 t^2+1=0 & I_{4}\\
 t=-1 & I_{4}
\end{array}
\end{equation}
We can easily find two linearly independent points:
\begin{align*}
P_{1}&=(0,fgh)\\
P_{2}&=(-f^2+g^2,\sqrt{2}f^2 g)
\end{align*}
with height pairing matrix 
\begin{equation}
\left(\begin{array}{cc}
 4 & 0\\
 0 & 2
\end{array}\right)
\end{equation}
and they span a rank $2$ subgroup which is of finite index in the full Mordell-Weil group.

\textbf{Type $(1,-1,1)$}.The Kodaira types of singular fibers that appear in this case are:
\begin{equation}
\begin{array}{rl}
\textrm{point} & \textrm{fiber type}\\
\hline
 t=\infty & I_{4}\\
 t=0 & I_{4}\\
 t^2+1=0 & I_{4}\\
 t^4 + 6t^2 + 1=0 & I_{2}
\end{array}
\end{equation}
We find two points:
\begin{align*}
P_{1}&=(0,\sqrt{-1}fgh)\\
P_{3}&=(-1/2 f^2, 1/(2\sqrt{-2}) f (g^2+h^2))\\
\end{align*}
with height pairing matrix 
\begin{equation}
\left(\begin{array}{cc}
 4 & 0\\
 0 & 2
\end{array}\right)
\end{equation}
The points $P_{1}$ and $P_{3}$ are linearly independent.

\subsection{Case $\alpha \beta - \alpha \gamma - \beta \gamma=0$}
The equation $\alpha \beta - \alpha \gamma - \beta \gamma=0$ is homogeneous and we can parametrize it as a quadric in $\mathbb{P}^{2}$. We consider the triples $(\alpha,\beta,\gamma)\in U$ that are parametrized by $r\neq 0$ as follows
\begin{equation}\label{eq:parameters_quadr}
\alpha=r(r-1),\beta=r,\gamma=r-1.
\end{equation}
The curve $E_{(\alpha,\beta,\gamma)}$ in this case has the Weierstrass equation
\begin{equation}\label{eq:curve_I0_star}
y^2=(x+r(r-1)(t^2-1)^2)(x+r(4t^2))(x+(r-1)(t^2+1)^2) 
\end{equation}
with discriminant $\Delta=16r^2(r-1)^2\delta^6$, where $\delta =-1 + r - 2 t^2 - 2 r t^2 - t^4 + r t^4$. It can be easily transformed into another form
\[(-1)\delta (y')^2=x'(x'-1)(x'-r)\]
where $x+r(4t^2)=x'(-1)\delta$ and $y=y'\delta^2$. The curve $(y')^2=x'(x'-1)(x'-r)$ determines a constant fibration so every fiber is of type $I_{0}$ and the twist by $-\delta$ introduces four $I_{0}^{*}$ fibers over the points $t_{0}$ such that $\delta(t_{0})=0$. The curve \eqref{eq:curve_I0_star} is hence a globally minimal Weierstrass model. Application of Shioda-Tate formula allows us to compute the bound $\rank E_{(\alpha,\beta,\gamma)}(\overline{\mathbb{Q}}(t))\leq 2$.

We observe that the equation $-\delta=s^2$ with respect to variables $t$ and $s$ defines a quartic model of elliptic curve $(y')^2=x'(x'-1)(x'-r)$. This means that the isotrivial elliptic surface attached to \eqref{eq:curve_I0_star} is a Kummer fibration $\mathcal{E}$ related to two elliptic curves $E_{1}:-\delta=s^2$ and $E_{2}:(y')^2=x'(x'-1)(x'-r)$. They are isomorphic over $\overline{\mathbb{Q}}$ and by \cite[\S 12.6]{Shioda_Schutt} it follows that
$\rho(\mathcal{E})=18+\textrm{rank}\mathop{Hom}(E_{1},E_{2})$. This implies that we have precisely rank $2$ if $E_{1}$ has complex multiplication, and otherwise rank $1$. We have the usual point of infinite order in any case on the curve \eqref{eq:curve_I0_star}.
\begin{align*}
P_{1}&=(0,2 (r-1) r t \left(t^4-1\right))
\end{align*}
The other linearly independent point in the CM case can be obtained as follows. We consider a quadratic twist of $E_{2}$ by $-\delta$. We denote this curve by $E_{2}^{(-\delta)}$. Let $\sigma$ denote the automorphism of the field $\overline{\mathbb{Q}}(t)(\sqrt{-\delta})$ determined by $\sigma(\sqrt{-\delta})=-\sqrt{-\delta}$. There is a natural isomorphism $\phi$ between the group of points $\tilde{H}=\{P:E_{2}(\overline{\mathbb{Q}}(t)(\sqrt{-\delta})):\sigma(P)=-P\}$ and the group $E_{2}^{(-\delta)}(\overline{\mathbb{Q}}(t))$. For a point $P\in\tilde{H}$ we have $P=(\alpha,\beta\sqrt{-\delta})$ for $\alpha,\beta\in\overline{\mathbb{Q}}(t)$. Then $\phi(P)=(\alpha (-\delta),\beta (-\delta)^2)$. Since $E_{1}$ and $E_{2}$ are isomorphic, we assume for now that $E_{1}=E_{2}$ to simplify the notation. Since $E_{2}$ is a CM curve, then $\textrm{End}(E_{2})$ is isomorphic to an order $R_{K}$ in some quadratic imaginary extension $K$ of $\mathbb{Q}$. We pick an element $\omega$ of $R_{K}$ that is not in $\mathbb{Z}$. It induces an endomorphism $[\omega]\in \textrm{End}(E_{2})$. On the curve $E_{2}$ our point $P_{1}$ from \eqref{eq:curve_I0_star} induces a point
\[P_{1}'=\left(\frac{4rt^2}{-\delta},\frac{2r(r-1)t(t^2-1)(1+t^2)}{\delta\sqrt{-\delta}} \right).\]
It maps via $\phi$ to $P_{1}^{(-\delta)}$ on $E_{2}^{(-\delta)}$. We observe that for any point $Q\in E_{2}(\overline{\mathbb{Q}}(t)(\sqrt{-\delta}))$ the endomorphism $[\omega]$ is equivariant with respect to $\sigma$: $[\omega](\sigma(Q))=\sigma([\omega](Q))$. In particular
$[\omega](\sigma(P_{1}'))=\sigma([\omega](P_{1}'))$ and this implies $[\omega](P_{1}')\in\tilde{H}$, and via $\phi$ it corresponds to a point in $E_{2}^{(-\delta)}(\overline{\mathbb{Q}}(t))$ which we denote by abuse of notation by $[\omega](P_{1}^{(-\delta)})$. The points $P_{1}^{(-\delta)}$ and $[\omega](P_{1}^{(-\delta)})$ cannot be linearly dependent because of the choice of $[\omega]$, so they span a rank two subgroup in $E_{2}^{(-\delta)}(\overline{\mathbb{Q}}(t))$. We denote for now the curve $E_{2}$ by $E_{r}$. We have proved the following lemma.
\begin{lemma}
Let $(\alpha,\beta,\gamma)\in U$ be such that $\alpha \beta - \alpha \gamma - \beta \gamma=0$. Choose parameter $r$ like in \eqref{eq:parameters_quadr}. Then the following holds
\begin{equation*}
\rank E_{(\alpha,\beta,\gamma)}(\overline{\mathbb{Q}}(t))=\left\{\begin{array}{ll}
1 & \textrm{if } E_{r}\textrm{ does not have complex multiplication,}\\
2 & \textrm{otherwise.}
\end{array}\right.
\end{equation*}
\end{lemma}

\begin{example}
In several simple cases of complex multiplication on $E_{2}$ we can actually compute explicitly the point $[\omega](P_{1}^{(-\delta)})$ for some particular choice of $\omega$. We freely adopt the results of \cite[II, Prop. 2.3.1]{Silverman_book}

\begin{figure}
\begin{equation*}
 \begin{array}{cccc}
  K & R_{K} & \omega & j(E)\\
  \mathbb{Q}(\sqrt{-1}) & \mathbb{Z}[\sqrt{-1}] & 1+\sqrt{-1} & 1728\\
  \mathbb{Q}(\sqrt{-2}) & \mathbb{Z}[\sqrt{-2}] & \sqrt{-2} & 8000\\
  \mathbb{Q}(\sqrt{-7}) & \mathbb{Z}[\sqrt{-7}] & \frac{1+\sqrt{-7}}{2} & -3375
 \end{array}
\end{equation*}
\caption{Three elliptic curves $E$ with j-invariant $j(E)$ and endomorphism ring isomorphic to $R_{K}$} 
\end{figure}
\begin{figure}
\begin{itemize}
\item $E: y^2=x^3+x,\quad j=1728, \quad \omega=1+\sqrt{-1}$
\[[\omega](x,y)=\left(\omega^{-2}(x+1/x),\omega^{-3} y(1-1/x^2)\right)\]
\item $E:y^2=x^3+4x^2+2x,\quad j=8000,\quad \omega=\sqrt{-2}$
\[ [\omega](x,y)=(\omega^{-2}(x+4+2/x),\omega^{-3} y (1-2/x^2) )\]
\item $E:y^2=x^3-35x+98,\quad j=-3375,\quad \omega =(1+\sqrt{-7})/2$
\[[\omega](x,y)=(\alpha^{-2}(x-7(1-\omega)^4/(x+\omega^2-2)),\omega^{-3} y (1+7(1-\omega)^4)/(x+\omega^2-2)^2  )  \]
\end{itemize}
\caption{Certain endomorphisms of degree $2$ in $\textrm{End}(E)$}\label{fig:Endo_examples}
\end{figure}
For a fixed parameter $r$ we choose an isomorphism between $E_{r}$ and curve $E$ from Figure \ref{fig:Endo_examples}. This induces via the chain of morphisms $E_{r}\xrightarrow{\textrm{iso}} E\xrightarrow{[\omega]} E\xrightarrow{\textrm{iso}^{-1}} E_{r}$ the endomorphism $[\omega]$ on $E_{r}$. We skip the simple but tedious algebraic manipulations and offer the final results

\begin{itemize}
 \item $f=2$, $j=1728$, $\eta=\sqrt{8\sqrt{-1}}$
 \[P_{2}^{(-\delta)}=\left(\frac{(1 - (2 - 4 \sqrt{-1}) t^2 + t^4)^2}{(1 + t^2)^2 (1 - 6 t^2 + t^4)}, 
 \frac{\eta t(-1 + 5t^2 - 26t^4 + 26t^6 - 5t^8 + t^{10})}{(1 + t^2)^3 (1 - 6 t^2 + t^4)^2}\right)\]
Height pairing matrix of $(P_{1}^{(-\delta)},P_{2}^{(-\delta)})$
\begin{equation*}
 \left(\begin{array}{cc}
  4 & 4\\
  4 & 8
 \end{array}\right)
\end{equation*}
\item $f=3+2\sqrt{2}$, $j=8000$, $\eta=\sqrt{-3 - 2\sqrt{2}}$
\[P_{2}^{(-\delta)}=\left(\frac{(\sqrt{2}+2) \left(t^4-1\right)^2}{4 t^2 \left(\sqrt{2} t^4+\sqrt{2}-4 t^2\right)},\frac{\eta \left(4 \sqrt{2} t^{10}-4 \sqrt{2} t^2-t^{12}-5 t^8+5 t^4+1\right)}{8 \sqrt{2} t^3 \left(-2 \sqrt{2} t^2+t^4+1\right)^2}\right)\]
Height pairing matrix of $(P_{1}^{(-\delta)},P_{2}^{(-\delta)})$
\begin{equation*}
 \left(\begin{array}{cc}
  4 & 0\\
  0 & 8
 \end{array}\right)
\end{equation*}
\item $f=1/2(1 + 3\sqrt{-7})$, $j=-3375$
\begin{align*}
x(P_{2}^{(-\delta)})&=\frac{(3 \sqrt{-7}-1) \left((-\sqrt{-7}-3) t^2+4 t^4+4\right)^2}{64 t^2 \left((3-\sqrt{-7}) t^4-3 (1-\sqrt{-7}) t^2-\sqrt{-7}+3\right)}\\
y(P_{2}^{(-\delta)})&= \frac{(\sqrt{-7}+1) \left(t^4-1\right) \left(3 (\sqrt{-7}-5) t^6+3 (\sqrt{-7}-5) t^2+4 t^8+24 t^4+4\right)}{t^3 \left(3 (\sqrt{-7}-5) t^2+8 t^4+8\right)^2}
\end{align*}

Height pairing matrix of $(P_{1}^{(-\delta)},P_{2}^{(-\delta)})$
\begin{equation*}
 \left(\begin{array}{cc}
  4 & 2\\
  2 & 8
 \end{array}\right)
\end{equation*}
\end{itemize}
\end{example}

\section{Case $\alpha=\beta=\gamma=1$}\label{sec:case_1_1_1}
The remaining case of triple $(\alpha,\beta,\gamma)=(1,1,1)$ was studied extensively in \cite{Naskrecki_thesis} and \cite{NaskreckiMW}. In this section we offer an alternative proof  of \cite[Lem. 5.8]{NaskreckiMW}. We apply theorems from \cite{Winkler_Rational} to avoid elaborate height computations performed in \cite{NaskreckiMW}. Moreover, the techniques applied in this section justify our previous restriction of family \eqref{eq:general_abc_family} to the case where $(f,g,h)=(t^2-1,2t,t^2+1)$. We conclude in this section that over $\overline{\mathbb{Q}}(t)$ the classification of the curves with parameters $(\alpha,\beta,\gamma)$ do not essentially depend on the choice of fixed parametrizing triple.
%\subsection{Field automorphisms and ranks} % 3 strony

Let $K$ be a field of characteristic $0$ and assume that
\[E:y^2=x^3+Ax^2+Bx\]
is a Weierstrass model of an elliptic curve such that $A,B\in K$. Let $\sigma:K\rightarrow K$ be an automorphism of field $K$. The curve
\[E^{\sigma}:y^2=x^3+\sigma(A)x^2+\sigma(B)x\]
is a Weierstrass model of another elliptic curve over $K$. The map
\begin{align}
E(K)&\rightarrow E^{\sigma}(K)\nonumber\\
(x,y)&\mapsto (\sigma(x),\sigma(y))\label{al:izomorfizm_nad_K}\\
O&\mapsto O\nonumber
\end{align}
establishes an isomorphism of the Mordell-Weil groups $E(K)$ and $E^{\sigma}(K)$.

Now we will recall certain results about rational curves and their parameterizations from \cite[Chap. 4]{Winkler_Rational}.
Let $F$ be any algebraically closed field of characteristic zero. The function $f(t)\in F(t)$ is in \emph{reduced form} when the numerator and denominator of $f$ are coprime.
\begin{definition}
Let $f\in F(t)$ be a rational function in reduced form. If $f$ is non-zero, the degree of $f$ is the maximum of the degrees of the numerator and denominator of $f$. When $f$ is zero, we define its degree to be $-1$. We denote the degree of $f$ by $\deg(f)$.
\end{definition}
Rational functions of degree $1$ are called \emph{linear}. When $f$ is linear, it is of the form $f(t)=(at+b)/(ct+d)$, where $ad-bc\neq 0$ and $a,b,c,d\in F$. We write shortly $f(t)=\gamma t$ where $\gamma$ is the corresponding matrix
\[\gamma=\left(\begin{array}{cc}
          a & b\\
          c & d
         \end{array}\right)\in \textrm{GL}_{2}(F).
\]
By a theorem of Clebsch, every irreducible curve of genus $0$ has a parametrization, cf. \cite[Chap. 4.1]{Winkler_Rational}. In particular, every smooth conic $C:\alpha a^2+\beta b^2=\gamma c^2$ has a parametrization. Let $F_{C}(x,y)=0$ be the affine equation of $C$. We always assume that $F_{C}$ is an irreducible and nonconstant polynomial in $F[x,y]$. The parametrization is a non-constant rational map
\[\mathcal{P}:t\mapsto (\chi_{1}(t),\chi_{2}(t))\]
such that $F_{C}(\chi_{1}(t),\chi_{2}(t))=0$. We define the degree of a parametrization $\mathcal{P}$ as follows
\[\deg\mathcal{P}=\max \{\deg \chi_{1}, \deg\chi_{2} \}.\]
We say that a parametrization $\mathcal{P}$ of curve $C$ is \emph{proper}, when the rational map $\mathcal{P}$ is birational. Two parametrizations of the same curve $C$ are related to each other.
\begin{theorem}[\protect{\cite[Chap. 4, Lemma 4.17]{Winkler_Rational}}]\label{thm:proper_params_change_coordinates}
Let $P$ be any affine parametrization of the rational curve $C$. Let $\mathcal{P}'$ be any other parametrization of $C$.
\begin{itemize}
 \item There exists a nonconstant rational function $f\in F(t)$ such that $\mathcal{P}'(t)=\mathcal{P}(f(t))$.
 \item Parametrization $\mathcal{P}'$ is proper if and only if there exists a linear function $f\in F(t)$ such that $\mathcal{P}'(t)=\mathcal{P}(f(t))$.
\end{itemize}
\end{theorem}
Now it is important for us to establish a relation between the degree of the defining polynomial $F_{C}(x,y)$ of our rational curve and the degree of any proper parametrization $\mathcal{P}$.

\begin{theorem}\label{theorem:params_main}
Let $C$ be an affine rational curve defined over $F$ with defining polynomial $F_{C}(x,y)\in F[x,y]$ and let $\mathcal{P}=(\chi_{1},\chi_{2})$ be a parametrization of $C$. Then $\mathcal{P}$ is proper if and only if
\[\deg \mathcal{P} = \max \{\deg_{x}(F_{C}),\deg_{y}(F_{C})\}.\]
Furthermore if $\mathcal{P}$ is proper and $\chi_{1}$ is nonzero, then $\deg\chi_{1} = \deg_{y}(F_{C})$; similarly if $\chi_{2}$ is nonzero then $\deg\chi_{2} = \deg_{x}(F_{C})$.
\end{theorem}

This theorem easily implies that every proper parametrization $\mathcal{P}_{C}$ of a smooth conic $C$ is of degree $2$ and the other way around every triple of coprime polynomials $f,g,h\in F[t]$ and such that $f^2+g^2=h^2$, $\max\{\deg f,\deg g,\deg h\}=2$ determines a proper parametrization of the curve $a^2+b^2=c^2$. We denote such a parametrization by $\mathcal{P}_{f,g,h}$. Its equation in variable $t$ is given by
\[\mathcal{P}_{f,g,h}(t)=(f(t)/h(t),g(t)/h(t)).\]
Any two such parametrizations $\mathcal{P}_{f,g,h}$ and $\mathcal{P}_{f',g',h'}$ are related by a linear change of variable $\mathcal{P}_{f',g',h'}(t)=\mathcal{P}_{f,g,h}(\gamma t)$ for a $\gamma\in\textrm{GL}_{2}(F)$. We denote by $E_{f,g,h}$ the curve in the form \ref{eq:general_abc_family} with $\alpha=\beta=\gamma$.
\begin{corollary}\label{cor:rank_comparison}
Let $(f,g,h)$ and $(f',g',h')$ be two triples of polynomials in variable $t$ that parametrize the conic $a^2+b^2=c^2$ in a proper way. There exists a linear function $\gamma t$, $\gamma\in\textrm{GL}_{2}(F)$ such that the automorphism $\sigma:t\mapsto \gamma t\in \textrm{Aut}(F(t))$ induces an isomorphism of the Mordell-Weil groups
\[E_{f,g,h}(F(t))\cong E_{f',g',h'}(F(t))\]
where $E_{f',g',h'}$ is $F(t)$-isomorphic to the curve $E_{f,g,h}^{\sigma}$. In particular, we obtain the equality
\[\textrm{rank }E_{f,g,h}(F(t))=2.\]
\end{corollary}
\begin{proof}
The curve $E_{f,g,h}$ is isomorphic over $F(t)$ to the curve $y^2=x(x-1)(x-(f/g)^2)$. Next we apply Theorem \ref{theorem:params_main} to compare parametrizations $\mathcal{P}_{f,g,h}$ and $\mathcal{P}_{f',g',h'}$. By the theorem there exists an element $\gamma\in\textrm{GL}_{2}(F)$ such that 
\[f/h(\gamma t)=f'/h'(t),\quad g/h(\gamma t)=g'/h'(t).\]
This easily implies that $f/g(\gamma t)=f'/g'(t)$. We apply the automorphism $\sigma$ to the curve $E_{f,g,h}$ and we get a curve $E_{f,g,h}^{\sigma}$ which is $F(t)$-isomorphic to $E_{f',g',h'}$. Finally this implies $E_{f,g,h}(F(t))$ is isomorphic to $E_{f,g,h}^{\sigma}(F(t))$ and to $E_{f',g',h'}(F(t))$. The last statement of the theorem follows for example from the fact that \[\textrm{rank }E_{t^2-1,2t,t^2+1}(F(t))=2\]
cf. \cite[Lemma 3.8]{Naskrecki_Acta}.
\end{proof}

\begin{remark}
We stress the fact that in general the curves $E_{f,g,h}$ and $E_{f',g',h'}$ for different parametrizations are not $F(t)$-isomorphic. This can be easily seen by comparing the $j$-invariants for both curves as a function of variable $t$.
\end{remark}

\medskip\noindent
We can keep track of the field of definition for the curve $E_{f,g,h}$. Suppose we let $F$ be a number field and assume $f,g,h$ all lie in $F[t]$ and that they determine a proper parametrization $\mathcal{P}_{f,g,h}$ of the conic $a^2+b^2=c^2$. We assume that the polynomials are coprime. We call a parametrization determined by polynomials $(t^2-1,2t,t^2+1)$ a \emph{standard parametrization}. It is easy to see that for any pair of coordinates $(x_{0},y_{0})\in F^2$ where $y_{0}\neq 0$ we put $t=\frac{x_{0}+1}{y_{0}}$ to recover the point. For the point $(-1,0)$ we put $t=0$ and for $(1,0)$ we put $t=1/s$ to change the coordinate chart and in the projective coordinates we take $s=0$. So every point of $\mathbb{P}^{1}(F)$ can be reached by this parametrization. We can assume without loss of generality that $\deg f=2$. If not, then $\deg g=2$ or otherwise this will imply $\deg h=1$ and the triple $(f,g,h)$ would not determine a proper parametrization of the conic.

\medskip\noindent
Then from the equation $f^2=h^2-g^2$ we can deduce that $h-g=s_{1}^{2}$ and $h+g=s_{2}^{2}$ where $f$ factors as $s_{1}s_{2}$. We must have $\deg s_{1}=\deg s_{2}=1$, otherwise the parametrization could not be proper. Assume $s_{1}=c(t-\alpha)$ and $s_{2}=d(t-\beta)$ for some $c,d,\alpha,\beta\in \bar{F}$. From the assumptions about $f,g,h$ we get that $cd\in F$. Because $h=(s_{1}^{2}+s_{2}^{2})/2\in F[t]$, $g=(s_{2}^{2}-s_{1}^{2})/2\in F[t]$, then $c^2,d^2\in F$ and $\alpha,\beta\in F$. We will show now that the parametrization $\mathcal{P}_{f,g,h}$ is related to the standard parametrization $\mathcal{P}_{t^2-1,2t,t^2+1}$ via a linear rational function with coefficients in $F$. This will imply that the groups $E_{f,g,h}(F(t))$ and $E_{t^2-1,2t,t^2+1}(F(t))$ are $F(t)$-isomorphic. 

\medskip\noindent
We have the equalities
\[\frac{f}{h}=\frac{2s_{1}s_{2}}{s_{1}^2+s_{2}^{2}}=\frac{2(\gamma t)}{1+(\gamma t)^2},\]
\[\frac{g}{h}=\frac{s_{2}^{2}-s_{1}^{2}}{s_{1}^2+s_{2}^{2}}=\frac{(\gamma t)^2-1}{1+(\gamma t)^2}\]
where $\gamma$ is a matrix from $\textrm{GL}_{2}(F)$
\[\gamma=\left(\begin{array}{cc}
   d/c & -(d/c)\beta\\
   1 & -\alpha
  \end{array}\right).\]
So the automorphism $t\mapsto \gamma t$ of $F(t)$ induces the isomorphism of the Mordell-Weil groups. The group $E_{t^2-1,2t,t^2+1}(F(t))$ can have rank $1$ or $2$ by \cite[Thm. 3.1]{NaskreckiMW}, and the result follows for any curve $E_{f,g,h}$ over $F(t)$ where $f^2+g^2=h^2$ determines a proper parametrization. 

\begin{remark}
This way of reasoning can be generalized to any conic $\alpha a^2+\beta b^2=\gamma c^2$ defined over $F$. We have to fix one parametrization of this conic over $F$ and then relate other $F(t)$-parametrizations to the fixed one. So in general we would also get a choice between rank $1$ or $2$.
\end{remark}

\begin{remark}
It is not always true that the equality $2\deg f=\deg (f^2-g^2)$ holds as the standard parametrization $(t^2-1,2t,t^2+1)$ might suggest. In fact we can have $\deg(f^2-g^2)<2\deg f$ and in the situation of our previous lemma this can only happen when $\deg f=\deg g=2$ (then we can only have $\deg(f^2-g^2)=3$). Consider the following example:
\begin{align*}
f & =\frac{t^2}{\sqrt{2}}-\frac{t}{\sqrt{2}}+i \left(\frac{t}{\sqrt{2}}-\frac{1}{2 \sqrt{2}}\right)+\frac{1}{2 \sqrt{2}}\\
g &=\frac{t^2}{\sqrt{2}}-\frac{t}{\sqrt{2}}+i \left(\frac{1}{2 \sqrt{2}}-\frac{t}{\sqrt{2}}\right)+\frac{1}{2 \sqrt{2}}\\
h &=t^2 - t\\
f^2-g^2 &=i \left(2 t^3-3 t^2+2 t-\frac{1}{2}\right)
\end{align*}
\end{remark}

\begin{remark}
We can always pull back an elliptic surface $(S,\mathbb{P}^{1},\pi)$ along a morphism $f:\mathbb{P}^{1}\rightarrow\mathbb{P}^1$ to obtain a new elliptic surface $(S',\mathbb{P}^1,\pi')$. Surface $S'$ is birational to the fiber product $S\times_{\phi}\mathbb{P}^{1}$. When we apply this construction to an automorphism $f\in\textrm{Aut}(\mathbb{P}^{1})$, surfaces $S$ and $S'$ are isomorphic. This implies that for different parametrizations $(f,g,h)$ and $(f',g',h')$ that both determine a proper parametrization of the Pythagorean conic, the surfaces attached to $E_{f,g,h}$ and $E_{f',g',h'}$ are in fact isomorphic. This isomorphism do not respect the elliptic fibrations. Nonetheless, the N\'{e}ron-Severi lattice is the same for both.
\end{remark}

\begin{remark}
From \cite[Prop. 11.14]{Shioda_Schutt} or \cite[Thm. 8.12]{Shioda_Mordell_Weil} we obtain the invariance of the height pairing under the automorphism $\sigma\in\textrm{Aut}(F(t))$. More precisely, for any $P,Q\in E(F(t))$ the intersection pairing $\langle P,Q\rangle_{E} $ equals $\langle P^{\sigma},Q^{\sigma}\rangle_{E^{\sigma}}$. In particular, this implies that the Mordell-Weil lattices on both curves are the same. 
\end{remark}

For polynomials $f,g,h$ that properly parametrize the conic $a^2+b^2=c^2$ the attached K3 surfaces corresponding to $E_{f,g,h}$ are isomorphic over $\overline{\mathbb{Q}}$ but not for any parametrization of $a^2+b^2=c^2$. In fact, we can easily produce an improper parametrization $(\frac{t^{2k}-1}{2},t^{k},\frac{t^{2k}+1}{2})$ for any $k\geq 2$. The elliptic surface attached to such a curve $E_{k}$ will have Euler characteristic equal to $2k$, so for different values of $k$ we will certainly obtain non-isomorphic elliptic surfaces. A quick computation reveals that the two linearly independent points in $E_{k}(\overline{\mathbb{Q}}(t))$ that we are able to produce might not give a complete list of free generators of the Mordell-Weil group. A numerical computation using the Nagao statistics, cf. \cite{SilvRosen} suggests that at least over $\mathbb{Q}(t)$ the Mordell-Weil rank should be again equal to $1$.
\begin{question}
Is it possible to determine the Mordell-Weil rank of the group $E_{k}(\mathbb{Q}(t))$ or $E_{k}(\overline{\mathbb{Q}}(t))$ when $k$ varies?
\end{question}
%\todo{Info about line parametrization fixed at a point $(x_{0},y_{0})$}

\section{Supersingular reduction}\label{sec:supersingular}
An elliptic curve of the form $E_{f,g,h}:y^2=x(x-f^2)(x-g^2)$ such that $f,g,h$ determine a proper parametrization $\mathcal{P}_{f,g,h}$ of the conic $a^2+b^2=c^2$ is a generic fiber of a K3-surface of a very special type. 
\begin{lemma}
Let $f,g,h\in \overline{\mathbb{Q}}[t]$ be polynomials that determine a proper parametrization of the conic $a^2+b^2=c^2$. Let $(\mathcal{E},\mathbb{P}^{1},\pi)$ be the Kodaira-N\'{e}ron model of the curve $E_{f,g,h}$. Then the triple $(\mathcal{E},\mathbb{P}^{1},\pi)$  is a singular elliptic K3-surface, i.e. its Picard number equals $20$. 
\end{lemma}
\begin{proof}
We have an explicit description of the Kodaira types corresponding to the singular fibers of $\pi$. Corollary \ref{cor:rank_comparison} implies that the rank of the Mordell-Weil group $E_{f,g,h}(\bar{\mathbb{Q}}(t))$ is two and the upper bound for the Picard number equal to $20$. Application of the Shioda-Tate formula allows us to conclude the statement. 
\end{proof}

Assume that $X$ is a K3 surface in characteristic $0$. We consider the situation when the rank of $NS(X)$ is maximal possible, equal to $20$. We denote by $d(X)$ the discriminant of the N\'{e}ron-Severi lattice. For elliptic K3 surfaces it can be computed if we have the information about the structure of the Mordell-Weil group of the generic fiber and about the fibration, cf. \cite{Shioda_Mordell_Weil}. We have an explicit formula, cf. \cite[\S 11.9]{Shioda_Schutt}
\[d(X)=(-1)^{\textrm{rank}(E_{\textrm{gen}}(K))}\textrm{disc}(\textrm{Triv}(X))\cdot \textrm{disc}(MWL(X))/(\#(E_{\textrm{gen}}(K))_{\textrm{tors}})^2\]
Field $K$ is the function field of $\mathbb{P}^{1}$ and $E_{gen}$ is the generic fiber of $X$. The lattice $\textrm{Triv}(X)$ is generated by the general smooth fiber $F$, image of the zero section $\overline{O}$ and the components of bad fibers that form root sublattices of standard types $A_{n}$, $D_{n}$ and $E_{n}$ with appropriate Dynkin diagrams. The lattice $MWL(X)=E_{gen}(K)/(E_{gen}(K))_{\textrm{tors}}$ with intersection pairing induced from the height pairing on $E_{gen}$. 

The Neron-Severi group $NS(X)$ embeds as a lattice into $H^{2}(X,\mathbb{Z})$ with its lattice structure inherited from the cup-product. The orthogonal complement of $NS(X)$ is called the transcendental lattice of $X$ and is denoted by $T(X)$. For K3 surfaces the second Betti number equals $22$, so for $X$ singular, this means that $\textrm{rank }T(X)=2$. It can be proved that $T(X)$ is an even lattice of discrminant $-d(X)$, cf. \cite{Shimada}.
Every singular K3-surface can be defined over a number field $F$ and we would like to consider the situation when $X$ can be reduced modulo a prime $p$. More precisely, we consider a non-empty open subset $U$ of $\textrm{Spec }\mathbb{Z}_{F}$, where $\mathbb{Z}_{F}$ is the ring of algebraic integers in $F$, and a smooth proper morphism $\mathcal{X}\rightarrow U$ with generic fiber isomorphic to $X$. We denote by $\pi_{F}$ the canonical morphism $\Spec\mathbb{Z}_{F}\rightarrow\Spec\mathbb{Z}$.  For a closed point $\mathfrak{p}$ in $U$ we denote by $X_{\mathfrak{p}}$ the fiber above $\mathfrak{p}$, which is a K3 surface defined over the residue field of $\mathfrak{p}$. In characteristic $p$ the rank of N\'{e}ron-Severi group can achieve rank $22$. Such a K3 surface is called \emph{supersingular}. We analyze the set $S_{p}(\mathcal{X})$ that contains primes $\mathfrak{p}$ above $p$ such that $X_{\mathfrak{p}}$ is supersingular. By the work of Shimada \cite{Shimada} we can now say when the reduction of the K3 surface would be supersingular. Let $\chi_{p}(x)$ denote the Legendre symbol $(x/p)$ for $p$ an odd prime. We have the following theorem
\begin{theorem}[\protect{\cite[Thm. 1]{Shimada}}]\label{theorem:Shimada_crit}
Let $p$ be a prime such that $p\nmid 2d(X)$. Then
\begin{itemize}
 \item if $\chi_{p}(x)=1$, then $S_{p}(\mathcal{X})=\emptyset$,
 \item if $\mathfrak{p}\in S_{p}(\mathcal{X})$, then $d(X_{\mathfrak{p}})=p^2$, i.e. the surface $X_{\mathfrak{p}}$ is of Artin invariant $1$.
\end{itemize}
Moreover, there exists a finite set $N$ of primes in $\mathbb{Z}$ that contains the prime divisors of $2d(X)$ such that for any $p\notin N$
\begin{equation}
S_{p}(\mathcal{X})=\left\{\begin{array}{ll}
                    \emptyset & \textrm{if }\chi_{p}(d(X))=1\\
                    \pi^{-1}(p) & \textrm{if }\chi_{p}(d(X))=-1\\
                   \end{array}\right.
\end{equation}
\end{theorem}
This theorem allows us to easily detect for which primes we can expect to have the supersingular reduction of the surface $S_{f,g,h}$ attached to the curve $E_{f,g,h}$. We need to compute the discriminant.

\begin{lemma}
Let $(f,g,h)$ be a triple of polynomials which parametrizes the conic $a^2+b^2=c^2$ in a proper way. The discriminant of the elliptic surface $S_{f,g,h}$ with generic fibre $E_{f,g,h}$ is equal to $-32$.
\end{lemma}
\begin{proof}
Let $X$ denote the elliptic surface $S_{f,g,h}$. From the assumptions the polynomials $f$, $g$ and $f^2-g^2$ are separable. The trivial lattice in $\NS(X)$ contains $\deg f+\deg g$ copies of the root lattice $A_{3}$, $\deg(f^2-g^2)$ copies of the lattice $A_{1}$ and possibly a copy of $A_{n-1}$ lattice that corresponds to the fiber above $\infty$, where $n=8\deg f-4\deg g-2\deg(f^2-g^2)$. If $\deg g=\deg f=2$ and $\deg(f^2-g^2)=2\deg f$ there is no such lattice, namely $n=0$. In the situation $\deg g=1$ we have $n=4$ and for $\deg g=2$ and $\deg(f^2-g^2)=3$ we get $n=2$. In each case we obtain that the trivial lattice is a sum $U+4A_{3}+4A_{1}$ where $U=\textrm{span }(F,\overline{O})$.

Torsion subgroup $E_{f,g,h}(\bar{\mathbb{Q}}(t))$ was computed in \cite{Naskrecki_thesis} and is of order $8$. The remaining part is the Mordell-Weil lattice. It is of rank two with generators
\begin{align*}
Q_{1}=&(-(1+\sqrt{2})g(g-h),\sqrt{-1}(1+\sqrt{2})g(g-h)(\sqrt{2}g-h)),\\
Q_{2}=&((f-h)(g-h),(f+g)(f-h)(g-h))
\end{align*}
and height pairing matrix
\begin{equation}
\left(\begin{array}{cc}
 1/2 & 0\\
 0 & 1
\end{array}\right)
\end{equation}
The root lattices corresponding to different bad fibers are orthogonal to each other in $\NS(X)$ and we have the standard formula $\disc(A_{n})=(-1)^{n}(n+1)$. The image of the zero section and the general fiber span a lattice $U$ in $\NS(X)$ which is of discriminant $-1$. Hence we get $d(X)=-32$.
\end{proof}
For a model of $S_{f,g,h}$ with good reduction at a prime above $p$ it easy is to check whether the reduction will be supersingular. We assume that $p\neq 2$, because this is exactly the condition $p\nmid 2d(X)$ from Theorem \ref{theorem:Shimada_crit}.
\[\chi_{p}(x)=-1\Leftrightarrow \left(\frac{-32}{p}\right)=-1\Leftrightarrow p\equiv 5,7,13,15(\textrm{mod }16).\]
It is now convenient to switch to the standard parametrization $(t^2-1,2t,t^2+1)$ and its associated elliptic surface which is defined over $\mathbb{Q}$. The Weierstrass equation for this model has the form
\[y^2=x(x-(t^2-1)^2)(x-4t^2).\]
It is a globally minimal model and is defined over $\mathbb{Z}[t]$. This has the advantage that we can perform the Tate algorithm both in characteristic zero and in positive characteristic (at least equal to $5$). We can do the blow-ups simultaneously at least if the reduced equation behaves in a similar way as the equation in characteristic $0$, cf. \cite[Tw. 2.2.12]{Naskrecki_thesis}. 

\begin{proposition}
The surface $S_{t^2-1,2t,t^2+1}$ has good reduction for primes $p\geq 5$.
\end{proposition}
\begin{proof}
We check that after modulo $p$ reduction the radicals of polynomials $a_{i}(t)$ and of discriminant of the Weierstrass equation, and the numerator and denominator of $j$-invariant remain separable (and that they do not have a common root modulo $p$). We also check this for the model at infinity. The support of the discriminants of all computed polynomials is contained in the set $\{2,3\}$. So we can perform the Tate algorithm in characteristic $p$ and in characteristic zero, and we will get the same reduction types, which implies that in fact we have a good reduction modulo $p$.
\end{proof}

The results lead to the following corollary.
\begin{corollary}
The surface $S_{t^2-1,2t,t^2+1}$ has good supersingular reduction at primes $p\geq 5$ such that $p\equiv 5,7,13,15(\textrm{mod }16)$. The discriminant of the N\'{e}ron-Severi group is equal to $p^2$. Generic fiber is an elliptic curve over $\mathbb{F}_{p}(t)$ with geometric Mordell-Weil rank $4$.
\end{corollary}
\begin{proof}
By \cite[Thm.3,4]{Schutt_Elkies} we can identify the action of Frobenius automorphism acting on the transcendental lattice $T(X)$. We find that in our situation we have the CM-form corresponding to level $8$ in \cite[Tab.1]{Schutt_CM}. This implies the equality $N=\{2,3\}$ where $N$ is the set from Theorem \ref{theorem:Shimada_crit}. Now we prove the last statement. Because of the good reduction situation we obtain exactly the same fiber types of bad reduction. For the supersingular case the Picard number equals $22$, so the application of Shioda-Tate formula leads to the conclusion that the Mordell-Weil rank of $E_{t^2-1,2t,t^2+1}(\overline{\mathbb{F}}_{p}(t))$ equals $4$.
\end{proof}

\begin{remark}
Supersingular K3 surface of Artin invariant $1$ is unique up to isomorphism and very special in another way that the generators of the N\'{e}ron-Severi group are defined over $\mathbb{F}_{p^{2}}$. So we can even say that $E_{t^2-1,2t,t^2+1}(\overline{\mathbb{F}}_{p}(t))=E_{t^2-1,2t,t^2+1}(\mathbb{F}_{p^2}(t))$.
\end{remark}

\begin{example}
We can produce explicit basis of the Mordell-Weil group for a few small primes. In fact, the theorem implies that the height of the extra points that add rank two to the Mordell-Weil group will grow unbounded and this means that the computations for sufficiently large primes require sieving over many rational functions with both denominator and numerator of very high degree. This computations become quickly infeasible for a typical computer.

In our computations we have exploited the fact that $E_{t^2-1,2t,t^2+1}[2]$ is contained in the $\mathbb{F}_{p}(t)$-rational points subgroup. Under this assumption we can apply a full $2$-descent described in \cite[Chap. 10]{Silverman_arithmetic}. The torsion subgroup of $E_{t^2-1,2t,t^2+1}(\mathbb{F}_{p^2}(t))$ is always isomorphic to $\mathbb{Z}/2\oplus\mathbb{Z}/4$. This is a consequence of \cite[Cor. 5.4]{NaskreckiMW} applied in positive characteristic. The crucial step is to prove that not all two-torsion points are divisible by $2$. It suffices to prove that the polynomial $f^2-g^2=t^4-6t^2+1$ is separable over $\mathbb{F}_{p^2}[t]$, which is always the case for primes $p\geq 3$.

The free part of the Mordell-Weil group will always contain the reductions of points $Q_{1}$ and $Q_{2}$, so below we only present the other two generators and compute the height pairing matrix. The trivial lattice $\textrm{Triv}(S_{t^2-1,2t,t^2+1})$ has discriminant $2^{12}$ and the torsion subgroup has order $8$. This implies that if we provide two points $Q_{3},Q_{4}$ such that the height pairing matrix of the tuple $(Q_{1},Q_{2},Q_{3},Q_{4})$ has determinant $p^2/2^6$, then the points generate the free part of the Mordell-Weil group.

\medskip\noindent
\textbf{Case $p=5$:}
We realize $\mathbb{F}_{5^2}$ as $\mathbb{F}_{5}[s]/(s^2+4s+2)$.
%\begin{equation*}
%Q_{3}=\left(\frac{s^{21}t(t+1)^2(t+4)}{(t+s^7)^2},\frac{s^8 t(t+1)^2(t+s^2)(t+s^{10})(t+4)(t+s^{14})}{(t+s^7)^3}\right)
%\end{equation*}
\begin{equation*}
Q_{3}=(s^3 t (t+1) (t+s^{22}), s^{10} t (t+1)(t+s^3)(t+s^{14})(t+s^{16})(t+s^{22}))
\end{equation*}
\begin{equation*}
Q_{4}=(t^4 + 4 t^2, t^5 + 4 t^3)
\end{equation*}
Height pairing matrix with determinant $5^2/2^6$.
\begin{equation*}
\left(\begin{array}{cccc}
 1/2 &   0 & 1/4 &   0\\
   0 &   1 &   0 & 1/2\\
 1/4 &   0 &   2 & -5/4\\
   0 & 1/2 & -5/4 & 3/2
\end{array}\right)
\end{equation*}
\medskip\noindent
\textbf{Case $p=7$:}
Let us assume that $\mathbb{F}_{7^2}=\mathbb{F}_{7}[s]/(s^2-3)$.
\begin{equation*}
Q_{3}=(t^2+t,s t(1+t)(2+t)^2)
\end{equation*}
\begin{equation*}
Q_{4}=(1,2 t (3 + t) (4 + t))
\end{equation*}
Height pairing matrix with determinant $7^2/2^6$.
\begin{equation*}
\left(\begin{array}{cccc}
 1/2 &   0 & -1/4 &   0\\
   0 &   1 &   0  & 1/2\\
-1/4 &   0 &   1  &  0\\
   0 & 1/2 &   0  &  2
\end{array}\right)
\end{equation*}

\medskip\noindent
\textbf{Case $p=13$:}
We realize $\mathbb{F}_{13^2}$ as $\mathbb{F}_{13}[s]/(s^2+12s+2)$.
\begin{equation*}
Q_{3}=(s^5 (t+s^{82})^2(t+12), s^{47} (t+s^4)(t+s^{18})(t+s^{82})(t+12)(t+s^{115}))
\end{equation*}
\begin{equation*}
Q_{4}=\left(\frac{11 t (t+2)^2 (t+6)^2}{(t+5)^2},\frac{3 t (t+2) (t+6) (t+7) (t+8) (t+11) \left(t^2+2 t+12\right)}{(t+5)^3}\right)
\end{equation*}
Height pairing matrix with determinant $13^2/2^6$.
\begin{equation*}
\left(\begin{array}{cccc}
1/2 &  0 & 1/4 &  0\\
  0 &  1 & 1/2 & 1/2\\
1/4 & 1/2 &  2 & 1/4\\
  0 & 1/2 & 1/4 & 7/2
\end{array}\right)
\end{equation*}
\end{example}

\begin{example}
The curve $y^2=(x-(t^2-1)^2)(x-4t^2)(x+(t^2+1)^2)$ that we considered before also determines a singular K3 surface. By the discriminant computation we checked that it is equal to $-2^5\cdot 3$, hence we obtain supersingular reduction at primes $p$ such that $\left(\frac{-2^5\cdot 3}{p}\right)=-1$, so $p= 13, 17, 19, 23, 37, 41, 43, 47, 61, 67, 71, \ldots$. We check that the attached CM-form attached to the transcendental lattice by \cite{Schutt_Elkies} is of level $N=24$, cf \cite[Tab. 1]{Schutt_CM}. For each supersingular prime we will obtain rank $8$ over $\mathbb{F}_{p^2}$. 
\end{example}

\section{Remarks}\label{sec:remarks}
Below we discuss several aspects of the general family \eqref{eq:general_family} that were not discussed elsewhere. We deal mainly with the family of type $\alpha=\beta=\gamma=1$ and for the other cases we can perform a similar study.
\subsection{Lower bounds over Q} %1.5 strony
We proved in \cite{NaskreckiMW} what are the lower bounds for the Mordell-Weil rank of specialization in our family. This theorem relies on the Silverman's specialization result, cf. \cite[III \S 11, Thm. 11.4]{Silverman_book}. Silverman's theorem allows us only to say that for all but finitely many elements in the number field $F$ over which the curve \eqref{eq:general_family} is defined, the specialization homomorphism will be injective. We will discuss below an approach to this problem, for number fields with class number one, that allows us to produce an infinite and explicit set of specialization for which the specialization homomorphism is injective.

\medskip\noindent
To simplify the exposition we will discuss only the specialization of curves $E_{f,g,h}$ for $f,g,h\in\mathbb{Q}[t]$ which parametrize the conic $a^2+b^2=c^2$. The tool we want to use is the theorem from \cite{Tadic}.

\begin{theorem}[\protect{\cite[Thm. 1.1]{Tadic}}]
Let $E$ be a nonconstant elliptic curve over $\mathbb{Q}(t)$ given by the equation
\[E=E(t): y^2=(x-e_{1})(x-e_{2})(x-e_{3})\quad (e_{1},e_{2},e_{3}\in\mathbb{Z}[t]).\]
Assume that $t_{0}\in\mathbb{Q}$ satisfies the following condition.

\begin{center}
(*) For every nonconstant square-free divisor $h$ in $\mathbb{Z}[t]$ of $(e_{1}-e_{2})(e_{1}-e_{3})$ or $(e_{2}-e_{1})(e_{2}-e_{3})$ or $(e_{3}-e_{1})(e_{3}-e_{1})$, the rational number $h(t_{0})$\\ is not a square in $\mathbb{Q}$.
\end{center}
Then the specialization homomorphism $\textrm{sp}_{t_{0}}:E(\mathbb{Q}(t))\rightarrow E(t_{0})(\mathbb{Q})$ is injective.
\end{theorem}
In our situation, let $e_{1}=0, e_{2}=(t^2-1)^2$ and $e_{3}=4t^2$. It is easy to compute the set of rational numbers that satisfies property $(*)$. We obtain the set that contains $35$ polynomials of degree at most $6$. Quick computation reveals that for all rational numbers $t_{0}$ with naive height smaller than $500$ (there are $304463
$ such rational numbers) about $97.5\%$ of this numbers satisfy condition $(*)$. The set of elements of $\mathbb{Q}$ of naive height at most equal to $10$ that satisfy condition (*) is
\begin{equation*}
\begin{split}
\{ -6, -10/3, -8/3, -7/4, -8/5, -10/7, -7/5, -7/6, -6/7, -5/7, -7/10,\\ -5/8, -4/7, -3/8, -3/10, -1/6, 1/6, 3/10, 3/8, 4/7, 5/8, 7/10, 5/7, 6/7,\\ 7/6, 7/5, 10/7, 8/5, 7/4, 8/3, 10/3, 6 \}.
\end{split}
\end{equation*}

Elements of $\mathbb{Q}$ that do not satisfy condition $(*)$ can still produce an injective specialization homomorphism or at least preserve the rank bound. In this case the rank of $E_{t^2-1,2t,t^2+1}(\mathbb{Q}(t))$ equals one and by a direct computation we have checked that for all $t_{0}$ for which the specialized curve is nonsingular, the rank was at least one for all $t_{0}$ with naive height at most $400$.
\subsection{Polynomial solutions}\label{subsec:uniform_a_b_c_model_par}
For any curve 
\[E_{\alpha,\beta,\gamma}:y^2=x(x-\alpha a^2)(x-\beta b^2),\quad \alpha a^2+\beta b^2=\gamma c^2\]
and a fixed conic $C:q(a,b,c)=0$ we can ask for the description of the $K(C)$-points of $E_{\alpha,\beta,\gamma}$. In fact, the curve $E_{\alpha,\beta,\gamma}$ is not well-defined over $K(C)$, so we slightly change the model to 
\[y^2=x(x-1)(x-\beta b^2/(\alpha a^2)).\]
By abuse of notation we will denote this curve again by $E_{\alpha,\beta,\gamma}$. The curve written this way is an elliptic curve defined over $K(C)$. We ask now for the description of $K(C)$ points on this curve, written explicitly in terms of homogeneous variables $a,b,c$. We will carefully analyze only the simplest case when $q(a,b,c)=a^2+b^2-c^2$ is the equation that defines Pythagorean triples. Function field $K(C)$ is isomorphic to $K(\mathbb{P}^{1})=\bar{\mathbb{Q}}(t)$, where $t$ is a variable.
We rewrite the equation 
\begin{equation}\label{eq:uniform_a_b_c_model}
y^2=x(x-1)(x-\frac{b^2}{a^2})
\end{equation}
in the form 
\[y^2=x(x-1)(x-\left(2t/(t^2-1)\right)^2).\]
We use the field isomorphism $\phi:K(C)\rightarrow K(\mathbb{P}^{1})$, which satisfies $\phi(a/c)=(t^2-1)/(t^2+1)$, $\phi(b/c)=2t/(t^2+1)$. This can be deduced from the standard parametrization of the circle by lines. The inverse to this map $\phi^{-1}$ satisfies $\phi^{-1}(t)=b/(c-a)$.

\begin{proposition}
Every $K(C)$-point $(x,y)$ on $E_{1,1,1}$ is represented by three polynomials $k,l,m\in \bar{\mathbb{Q}}[a,b,c]$ that satisfy $x=k/l^2$, $y=m/l^3$ and $\deg k=2+2\deg l$, $\deg m=3+3\deg l$.
\end{proposition}
\begin{proof}
The proof follows from the definition of $\phi$. We observe that each $K(C)$-point on $E_{1,1,1}$ is obtained from the point on \eqref{eq:uniform_a_b_c_model} by a linear change of variables $(x,y)\mapsto (x a^2, y a^3)$. 
\end{proof}

\subsection{Two isogeny} % 1 strona
We observe that the curve
\[E_{f,g}: y^2=x(x-f^2)(x-g^2)\]
admits two--isogenies defined over the field $K$ of definition of elements $f,g$. The isogenous curve $E_{f,g}/\langle T\rangle$ for $T\in E_{f,g}[2]$ is not always of the form above if the kernel contains only $(f^2,0)$ or by symmetry $(g^2,0)$. We analyze the isogeny with kernel $\{O,(0,0)\}$. It is easy to determine it explicitly by Velu formulas \cite{Velu} and an explicit computation in MAGMA. Consider the curve
\[E_{i(f-g),i(f+g)}: y^2=x(x+(f-g)^2)(x+(f+g)^2)\]
which is isomorphic to $E_{f,g}/\langle (0,0)\rangle$ over $K$ with isomorphism $(x,y)\mapsto (x+f^2+g^2,y)$. The two-isogeny $\tau: E_{f,g}\rightarrow E_{f,g}/\langle (0,0)\rangle$ is given by the formula
\[\tau(x,y)=((x^2 + f^2 g^2) / x , (x^2 y - f^2 g^2 y) / x^2).\]
If we allow the relation $f^2+g^2=h^2$, then $-(f-g)^2-(f+g)^2=-2h^2$. Both conics $a^2+b^2=c^2$ and $-a^2-b^2=-2c^2$ can be easily parametrized properly with polynomials in $\mathbb{Q}(t)$, however the curves are not $\mathbb{Q}$-isomorphic. The existence of $\mathbb{Q}(t)$-isogeny implies that we have exactly the same rank of Mordell-Weil  groups over $\mathbb{Q}(t)$ for both curves associated with those conics.

\section*{Acknowledgements}
The author would like to express his gratitude to the organisers of the ALANT 2014 conference for allowing him to speak about his results and stimulating atmosphere of the conference. He also thanks Remke Kloosterman for his suggestion of proof in \S \ref{subsec:uniform_a_b_c_model_par}. The author was supported by the National Science Centre Poland research grant 2012/05/N/ST1/02871 and by the DFG grant Sto299/11-1 within the framework of the Priority Programme SPP 1489.

\bibliography{bibliography}

\providecommand{\bysame}{\leavevmode\hbox to3em{\hrulefill}\thinspace}
\providecommand{\MR}{\relax\ifhmode\unskip\space\fi MR }
\providecommand{\MRhref}[2]{%
  \href{http://www.ams.org/mathscinet-getitem?mr=#1}{#2}
}
\providecommand{\href}[2]{#2}
\begin{thebibliography}{10}

\bibitem{Artin_SD_K3_Tate_Conjectures}
Michael Artin and Peter Swinnerton-{D}yer, \emph{The {S}hafarevich-{T}ate
  conjecture for pencils of elliptic curves on {K}3 surfaces}, Invent. Math.
  \textbf{20} (1973), 249--266.

\bibitem{MAGMA}
Wieb Bosma, John Cannon, and Catherine Playoust, \emph{The {M}agma algebra
  system. {I}. {T}he user language.}, J. Symbolic Comput. \textbf{24} (1997),
  no.~3-4, 235--265, Computational algebra and number theory ({L}ondon, 1993).

\bibitem{UlasBremner}
Andrew {Bremner} and Maciej {Ulas}, \emph{{Points at rational distances from
  the vertices of certain geometric objects}}, ArXiv e-prints (2015), 1--23,
  1502.07312.

\bibitem{Schutt_Elkies}
Noam~D. Elkies and Matthias Sch{\"u}tt, \emph{Modular forms and {K}3 surfaces},
  Adv. Math. \textbf{240} (2013), 106--131.

\bibitem{Esnault}
Hélène {Esnault}, Keiji {Oguiso}, and Xun {Yu}, \emph{{Automorphisms of
  elliptic {K}3 surfaces and Salem numbers of maximal degree}}, ArXiv e-prints
  (2014), 1--12, 1411.0769.

\bibitem{Fulton}
William Fulton, \emph{Intersection theory}, Ergebnisse der Mathematik und ihrer
  Grenzgebiete (3) [Results in Mathematics and Related Areas (3)], vol.~2,
  Springer-Verlag, Berlin, 1984.

\bibitem{Tadic}
Ivica Gusi{\'c} and Petra Tadi{\'c}, \emph{Injectivity of the specialization
  homomorphism of elliptic curves}, J. Number Theory \textbf{148} (2015),
  137--152.

\bibitem{Huybrechts_Geometry}
Daniel Huybrechts, \emph{Complex geometry}, Universitext, Springer-Verlag,
  Berlin, 2005.

\bibitem{Kloosterman_rank_15}
Remke Kloosterman, \emph{Elliptic {K}3 surfaces with geometric {M}ordell-{W}eil
  rank 15}, Canad. Math. Bull. \textbf{50} (2007), no.~2, 215--226.

\bibitem{Milne_Tate_Conjecture}
James Milne, \emph{On a conjecture of {A}rtin and {T}ate}, Ann. of Math. (2)
  \textbf{102} (1975), no.~3, 517--533.

\bibitem{Naskrecki_Acta}
Bartosz Naskręcki, \emph{Mordell-{W}eil ranks of families of elliptic curves
  associated to {P}ythagorean triples}, Acta Arithmetica \textbf{160} (2013),
  no.~2, 159--183.

\bibitem{Naskrecki_thesis}
\bysame, \emph{Ranks in families of elliptic curves and modular forms}, Adam
  Mickiewicz University (2014), Ph.D. thesis.

\bibitem{Naskrecki_EDS}
\bysame, \emph{Divisibility sequences of polynomials and heights estimates}, to
  appear in New York Journal of Mathematics (2016), 1--32.

\bibitem{NaskreckiMW}
\bysame, \emph{Mordell-{W}eil ranks of families of elliptic curves parametrized
  by binary quadratic forms}, preprint (2016), 1--21.

\bibitem{Oguiso_c2}
Keiji Oguiso, \emph{An elementary proof of the topological {E}uler
  characteristic formula for an elliptic surface}, Comment. Math. Univ. St.
  Paul. \textbf{39} (1990), no.~1, 81--86.

\bibitem{SilvRosen}
Michael Rosen and Joseph~H. Silverman, \emph{On the rank of an elliptic
  surface}, Invent. Math. \textbf{133} (1998), no.~1, 43--67.

\bibitem{Schutt_CM}
Matthias Sch{\"u}tt, \emph{C{M} newforms with rational coefficients}, Ramanujan
  J. \textbf{19} (2009), no.~2, 187--205.

\bibitem{Winkler_Rational}
J.~Rafael Sendra, Franz Winkler, and Sonia P{\'e}rez-D{\'{\i}}az,
  \emph{Rational algebraic curves}, Algorithms and Computation in Mathematics,
  vol.~22, Springer, Berlin, 2008, A computer algebra approach.

\bibitem{Shimada}
Ichiro Shimada, \emph{Transcendental lattices and supersingular reduction
  lattices of a singular {K}3 surface}, Trans. Amer. Math. Soc. \textbf{361}
  (2009), no.~2, 909--949.

\bibitem{Shimada_auto_K3}
Ichiro {Shimada}, \emph{{Automorphisms of supersingular {K}3 surfaces and Salem
  polynomials}}, ArXiv e-prints (2015), 1--16, 1503.04517.

\bibitem{Shioda_Mordell_Weil}
Tetsuji Shioda, \emph{On the {M}ordell-{W}eil lattices}, Comment. Math. Univ.
  St. Paul. \textbf{39} (1990), no.~2, 211--240.

\bibitem{Shioda_Schutt}
Tetsuji Shioda and Matthias Sch{\"u}tt, \emph{Elliptic surfaces}, ArXiv
  e-prints (2010), arXiv:0907.0298v3.

\bibitem{Silverman_arithmetic}
Joseph Silverman, \emph{The arithmetic of elliptic curves}, Graduate Texts in
  Mathematics, vol. 106, Springer-Verlag, New York, 1986.

\bibitem{Silverman_book}
\bysame, \emph{Advanced topics in the arithmetic of elliptic curves}, Graduate
  Texts in Mathematics, vol. 151, Springer-Verlag, New York, 1994.

\bibitem{Luijk_Heron}
Ronald van Luijk, \emph{An elliptic {K}3 surface associated to {H}eron
  triangles}, J. Number Theory \textbf{123} (2007), no.~1, 92--119.

\bibitem{Velu}
Jacques V{\'e}lu, \emph{Isog\'enies entre courbes elliptiques}, C. R. Acad.
  Sci. Paris S\'er. A-B \textbf{273} (1971), A238--A241.

\end{thebibliography}
\bibliographystyle{amsplain}

\end{document}